\documentclass[10pt]{article}
\usepackage{graphicx} %
\usepackage{amsmath}
\usepackage{amsthm}
\usepackage{amssymb}   
\usepackage{mathrsfs}  
\usepackage{enumitem}  
\usepackage[
    margin=1in
]{geometry}
\usepackage[colorlinks,
            linkcolor=red,
            anchorcolor=blue,
            citecolor=green
            ]{hyperref}
\usepackage{chngcntr}  
\numberwithin{equation}{section}
            
\usepackage{mathtools}

\usepackage{algorithm}
\usepackage{algorithmic}
\usepackage{natbib}

\usepackage{microtype}
\usepackage{graphicx}
\usepackage{subfigure}
\usepackage{booktabs} %
\usepackage{makecell}

\usepackage{hyperref}

\newcommand{\bx}{\mathbf{x}}
\newcommand{\by}{\mathbf{y}}
\newcommand{\bz}{\mathbf{z}}

\newcommand{\bb}{\mathbf{b}}

\newcommand{\bu}{\mathbf{u}}

\newcommand{\bv}{\mathbf{v}}

\newcommand{\bs}{\mathbf{s}}

\newcommand{\gtil}{\widetilde G}
\newcommand{\Et}{\mathbb{E}_t}
\newcommand{\fil}{\mathcal{F}}
\DeclareMathOperator{\gap}{gap}
\DeclareMathOperator{\Reg}{Reg}
\DeclareMathOperator{\proj}{proj}

\newcommand{\R}{\mathbb{R}}

\newcommand{\E}{\mathbb{E}}

\renewcommand{\phi}{\varphi}

\newcommand{\res}{\mathsf{Res}}

\DeclareFontFamily{OT1}{pzc}{}
\DeclareFontShape{OT1}{pzc}{m}{it}{<-> s * [1.200] pzcmi7t}{}
\DeclareMathAlphabet{\mathpzc}{OT1}{pzc}{m}{it}

\DeclareMathOperator{\dist}{dist}

\DeclareMathOperator{\dom}{dom}

\DeclareMathOperator{\prox}{prox}

\newcommand{\inprod}[2]{\left\langle{#1},{#2}\right\rangle}

\usepackage{xcolor,colortbl}

\hypersetup{
  linkcolor  = blue,
  citecolor  = teal,
  colorlinks = true,
}
\makeatletter
\newcommand\blfootnote[1]{%
  \begingroup
  \renewcommand{\@makefntext}[1]{\noindent\makebox[1.8em][r]#1}
  \renewcommand\thefootnote{}\footnote{#1}%
  \addtocounter{footnote}{-1}%
  \endgroup
}
\makeatother

\usepackage[capitalize,noabbrev]{cleveref}

\theoremstyle{plain}
\newtheorem{theorem}{Theorem}[section]

\newtheorem{lemma}[theorem]{Lemma}
\newtheorem{corollary}[theorem]{Corollary}
\theoremstyle{definition}

\newtheorem{assumption}[theorem]{Assumption}
\newtheorem{fact}[theorem]{Fact}

\crefname{fact}{Fact}{Facts}
\Crefname{fact}{Fact}{Facts}

\theoremstyle{remark}
\newtheorem{remark}[theorem]{Remark}

\crefname{assumption}{assumption}{assumptions}
\Crefname{assumption}{Assumption}{Assumptions}

\usepackage[textsize=tiny]{todonotes}

\title{Improving the Last-Iterate Guarantees of Anytime Algorithms for Stochastic Monotone Variational Inequalities}

\author{Jun-Hyun Kim\footnote{University of British Columbia, Vancouver. \url{junhyun@student.ubc.ca}} \and Ahmet Alacaoglu\footnote{Department of Mathematics, University of British Columbia, Vancouver. \url{ahmet.alacaoglu@ubc.ca}}}
\date{}

\begin{document}
\maketitle
\begin{abstract}
We analyze a stochastic algorithm with Halpern-type anchoring for constrained convex-concave problems and monotone variational inequalities. This single-loop and single-call algorithm uses one unbiased sample of the gradient operator at every iteration, to be applicable to monotone games with noisy feedback. With $t$ denoting the iteration counter, we prove an anytime last-iterate convergence rate of $O(t^{-1/4})$ for both the gradient-mapping norm and restricted gap, bypassing the $O(t^{-1/5})$ constrained-anytime bottleneck in the literature. Specializing then to multi-point oracles, we use variance reduction to achieve the $O(t^{-1/2})$ rate with an anytime single-loop algorithm using $2$ samples per iteration.
Our results allow constrained problems with a potentially unbounded feasible set; as well as a structured class of stochastic oracles whose variance need not be uniformly bounded.
\end{abstract}

\section{Introduction}
We focus on the unifying framework of monotone variational inequalities (VIs) that cover such problems as monotone games and convex-concave min-max problems. In particular, we focus on the specific inclusion describing this problem, given as 
\begin{equation}\label{eq:inclusion}
    0\in G(\bz^\star)+\partial r(\bz^\star),
\end{equation}
where $G\colon\mathbb{R}^d\to\mathbb{R}^d$ is monotone, $L$-Lipschitz; and $r\colon\mathbb{R}^d\to \mathbb{R}\cup\{+\infty\}$ is proper, convex, closed.
This is equivalent to the following form of the monotone VI, where the goal is to find $\bz^\star$ such that
\begin{equation}\label{eq: vi_def}
\inprod{G(\bz^\star)}{\bz-\bz^\star}+r(\bz)-r(\bz^\star)\geq0
\text{~for every~} \bz\in\R^d. 
\end{equation}
An important application of \eqref{eq: vi_def} is learning in monotone games \cite{cesa2006prediction}, where Nash equilibria can be characterized as solutions of a monotone VI, see, e.g., \cite{mertikopoulos2019learning}. In this setting, one action profile is played and one noisy payoff feedback is observed at each round. Hence, single-call methods fit this feedback model, whereas two-call methods require an additional gradient evaluation at a different action profile. 

Since the current action profile is the one actually played, the horizon is generally unknown, and we are interested in observing the dynamics of the learning process, last-iterate and anytime guarantees are particularly relevant \cite{cai2023doubly}. As a result, standard results for stochastic VI algorithms that show convergence rates on the averaged iterate or those that use increasing mini-batches, variance reduction, or multi-loop algorithms are not suitable for this setting. Motivated by this setup, in Section \ref{sec: alg_results}, we study an anytime, single-call algorithm with last iterate guarantees.

Another classical application outside of the learning in games setting is constrained optimization
\begin{align}\label{eq: lin_cons_prob}
    \min_{\bx} f(\bx)\colon A\bx\leq \bb  \iff \min_{\bx}\max_{\by\geq 0} f(\bx) + \langle A\bx-\bb, \by \rangle,
\end{align}
where single-call algorithms are not necessary and variance reduced methods are applicable. In particular, we often have a large matrix $A$ from which we sample rows or columns to get unbiased estimates for gradients, which is amenable to the oracle access for variance reduction. An important aspect in this setting is that we need algorithms that can be applied for unbounded constraint sets since neither the primal nor the dual domains are bounded in \eqref{eq: lin_cons_prob}. In Section \ref{sec: storm}, we study a single-loop algorithm suitable for the optimization setting by using variance reduction with $2$ calls per iteration.
\paragraph{Optimality measures.} Two standard optimality measures we will consider are gradient mapping norm, also known as the natural residual, and the restricted gap function \cite{facchinei2003finite}. In particular,
for a fixed $\rho>0$, define the gradient mapping by
\begin{align}\label{eq: gradmap}
    \mathcal G_\rho(\bz)
    =\frac{1}{\rho}\big(\bz-\prox_{\rho r}(\bz-\rho G(\bz))\big). 
\end{align}
The norm $\|\mathcal G_\rho(\bz)\|$ vanishes if and only if
\eqref{eq:inclusion} holds at $\bz$. For simplicity we will set $\rho=\frac{1}{L}$ throughout.
This is an optimality measure generalizing the gradient norm in the unconstrained case and is suitable for problems with unbounded feasible sets, in contrast to the gap we introduce next.

We present the gap function for the special case of $r(\bz) = \delta_{Z}(\bz)$ where $\delta_Z$ is the indicator function for a closed and convex set $Z$:
\begin{align}\label{eq: gap_def}
    \gap(\bz)=\max_{\bu\in Z}\left\{ \langle G(\bz), \bz-\bu \rangle\right\}.
\end{align}
Since this quantity is not suitable when set $Z$ is not bounded, restricted versions of gap functions are often used \cite{nesterov2007dual}. 
For restricted gap functions, they generally take the maximum over a compact set $\mathcal{U}$ and this set needs to satisfy certain requirements for restricted gap functions to be valid optimality measures \cite[Lemma 1]{nesterov2007dual}. 
For example, the set needs to contain the output of the algorithm and a solution. Since it is generally not possible to show the iterates stay bounded for stochastic algorithms, this optimality measure is less meaningful without a bounded domain or for stochastic algorithms. 
On the other hand, when the domain is bounded, such as in matrix games, then the gap functions are well-defined optimality measures.

To compare with some relevant works, we will prove rates for the gap function, which will be a direct consequence of the rates we prove on the natural residual/gradient mapping norm, given in \eqref{eq: gradmap}. The relationship between the gap function and the gradient mapping is formalized in Corollary \ref{cor: resgap} for completeness, which is a standard argument.
\begin{table*}[t!]
\centering
\scriptsize
\begin{tabular}{l l c c c c l l}
\toprule
\textbf{Method}
& \makecell{\textbf{Constraint}}
& \textbf{Anytime} &\textbf{Batch}& \makecell{\textbf{Variance}\\\textbf{Reduction}}&
& \makecell[l]{\textbf{Last-iterate rate}\\\textbf{and measure}}
& \makecell[l]{\textbf{Variance}\\\textbf{assumption}}
\\
\midrule
\cite{abe2025boosting} 
& Bounded
& $\times$ & $2$& $\times$&
& $\widetilde{O}(T^{-1/7})$, gap
& Bounded
\\[3mm]
\hline\\
\cite{zheng2026lastiterate} 
& Unbounded
& $\times$ & $2$ &$\times$&
& $O(T^{-1/4})$, gap
& Bounded$^\ddagger$
\\[3mm]
\cite{zheng2026lastiterate} 
& Unbounded
& $\checkmark$ &$2$&$\times$&
& $O(t^{-1/5})$, gap
& Bounded$^\ddagger$
\\[3mm]
\hline\\
\cite{ito2026last} 
& Bounded
& $\times$ & $1$&$\times$&
& $O(T^{-1/4})$, gap
& \makecell[l]{Bounded }
\\[3mm]
\cite{ito2026last} 
& Bounded
& $\checkmark$ & $1$&$\times$&
& $O(t^{-1/5})$, gap
& \makecell[l]{Bounded }
\\[3mm]
\hline\\
\cite{sohrabi2026accelerated} 
& Unconstrained
& $\checkmark$ & $1$&$\times$&
& $O(t^{-1/4})$, grad. norm
& \eqref{eq: goma_asp}
\\[3mm]
\hline\\
\textbf{Section \ref{sec: alg_results}}
& Unbounded
& $\checkmark$& $1$&$\times$&
& \makecell[l]{
$O(t^{-1/4})$, grad. mapping\\
$O(t^{-1/4})$, gap
}
& Blum-Gladyshev 
\\
\bottomrule
\end{tabular}
\label{tab: table1}
\caption{\small{Comparison of single-loop and anytime stochastic methods with last-iterate guarantees. The rates are stated with $T$ for fixed horizon and $t$ for anytime results. The  Blum-Gladyshev condition is given in \Cref{asp: 2}. $^\ddagger$\cite[Section 2.1]{zheng2026lastiterate} remarked that their analysis extends to a stronger variant of our \Cref{asp: 2}, but did not provide a proof.}}
\end{table*}
\begin{table*}[t!]
\centering
\scriptsize
\begin{tabular}{l l c c c l l}
\toprule
\textbf{Method}
& \makecell{\textbf{Constraint}}
& \textbf{Anytime} &\textbf{Batch}
& \makecell[l]{\textbf{Complexity}\\\textbf{and output}}
& \makecell[l]{\textbf{Variance}\\\textbf{assumption}}
\\
\midrule
\cite{cai2022stochastic} 
& Unconstrained
& $\times$ & incr.
& $\widetilde{O}(\varepsilon^{-3})$, last-iter.
& Bounded
\\[3mm]
\cite{tran2026unbiased} 
& Unbounded
& $\times$ & incr.
& $O(\varepsilon^{-10/3})$, random-iter.
& Bounded
\\[3mm]
\cite{alacaoglu2025towards} 
& Unbounded
& $\times$ & $2$
& $O(\varepsilon^{-4})$, random-iter.
& Blum-Gladyshev
\\[3mm]
\textbf{\textbf{Section \ref{sec: storm}}}
& Unbounded
& $\checkmark$& $2$
& 
$O(\varepsilon^{-2})$, last-iter.
& Blum-Gladyshev
\\
\bottomrule
\end{tabular}
\label{tab: table1}
\caption{\small{Comparison of single-loop methods with last-iterate guarantees and variance reduction.}}
\end{table*}
\paragraph{Related works.} 
In the deterministic case, optimal last iterate guarantees are obtained recently by anytime algorithms \cite{yoon2021accelerated,cai2022accelerated,cai2023doubly} by using anchoring; and suboptimal guarantees were shown in \cite{golowich2020last,gorbunov2022last} for standard VI algorithms, such as the classical extragradient method \cite{korpelevich1976extragradient}. In the game setting, last iterate guarantees are proven in \cite{golowich2020tight} for the optimistic gradient algorithm \cite{popov1980modification,rakhlin2013optimization}. In the stochastic case,
last iterate guarantees without using mini-batches or multi-loop algorithms has been of interest for learning in games with noisy feedback, and many works focused on this direction, including \cite{chen2024last,azizian2021last,abe2025boosting,sohrabi2026accelerated,zheng2026lastiterate,ito2026last,abe2023adaptively,hsieh2022no}.

The literature on the last iterate guarantees for anytime, single-loop algorithms using a single sample (or $2$ samples) at every iteration had recent interesting developments. In particular, for unconstrained problems \cite{sohrabi2026accelerated} showed the $O(t^{-1/4})$ rate for an anchored algorithm. For constrained problems, two parallel works showed the anytime rate $O(t^{-1/5})$ for the restricted gap function \cite{zheng2026lastiterate, ito2026last}. Surprisingly, both of these works concluded the following separation: when we let go of the \emph{anytime} requirement, that is, when the horizon $T$ is known in advance, it was possible to get the rate $O(T^{-1/4})$ for the last iterate where $T$ is the horizon; yet for the anytime case, both works got the worse $O(t^{-1/5})$ rate.

Other single-loop algorithms in the literature are not compatible with the requirements of our setting: for example the developments in \cite{cai2022stochastic} focused on unconstrained problems (for monotone and Lipschitz operators) and used variance reduction along with increasing batches; the works \cite{pethick2023solving,alacaoglu2026solving,alacaoglu2025towards} relied on variance reduction and multi-point oracles and did not prove last iterate guarantees. Classical works obtain optimal rates (in terms of $\varepsilon$ independence) on the gap when it is evaluated at the averaged iterate, see for example \cite{nemirovski2009robust}.

A large body of literature used increasing mini-batch sizes to obtain $O(\varepsilon^{-4})$ complexity \cite{iusem2017extragradient,pethick2023stable,kotsalis2022simple,lee2021fast}. Our complexity in Section \ref{sec: alg_results} matches this by avoiding mini-batches. When we allow increasing mini-batches or multi-loop algorithms, the complexity for residual-type optimality measures was improved to $O(\varepsilon^{-10/3})$ by using mini-batches and variance reduction in the work \cite{tran2026unbiased}, under \Cref{asp: 3}. Only very recently, methods based on multiple loops obtained the order-optimal $\widetilde{O}(\varepsilon^{-2})$ complexity for gradient norm for unconstrained problems \cite{chen2024near} and natural residual and tangent residual in the constrained case \cite{alacaoglu2026make,ji2026computation}. 

To our knowledge, the best complexity with a single-loop method for the residual remained at $O(\varepsilon^{-10/3})$. We obtain the best-known and order optimal rate $O(\varepsilon^{-2})$ under \Cref{asp: 3} with a single-loop algorithm using $2$ unbiased oracles at each iteration. \Cref{asp: 3} is stronger than the assumptions of \cite{alacaoglu2026make,ji2026computation}, however, unlike these works, we use a single-loop method without using increasing mini-batch sizes and multiple loops. \\[2mm]
\textbf{Context and Contributions. } The best anytime rate for the restricted gap stayed at $O\left(t^{-1/5}\right)$ \cite{zheng2026lastiterate,ito2026last}, whereas both works could obtain the horizon-dependent $O(T^{-1/4})$ rate where $T$ is the horizon. At the same time, an anytime $O(t^{-1/4})$ was known for unconstrained problems \cite{sohrabi2026accelerated, zheng2026lastiterate}. The difficulty for extending to the constrained case was discussed in \cite[Remark 4.2]{zheng2026lastiterate}; open questions to derive an anytime $O(t^{-1/4})$ rate for constrained problems were mentioned in \cite[Section 7]{sohrabi2026accelerated} and \cite[Section 6]{zheng2026lastiterate}. More technical details are provided in Section \ref{sec: pointwise}.

In this work, we prove that the separation from \cite{zheng2026lastiterate,ito2026last} can be avoided, that is, we prove the anytime rate $O\left(t^{-1/4}\right)$ for constrained problems. To allow unbounded constraint sets, we show rates on gradient mapping norm (in addition to the gap function), without the bounded variance assumption, by allowing variance to grow as fast as the displacement of the iterates from an initial point. Our algorithm is of stochastic gradient-type, with anchoring (in view of Halpern \cite{halpern1967fixed}), using a single-call of the unbiased operator at every iteration.

Next, we show that by incorporating the STORM variance reduction, we can improve the rate on the last iterate to $O(t^{-1/2})$. The algorithm is still single-loop and anytime, but not single-call anymore since it samples two unbiased oracles. This leads to the complexity $O(\varepsilon^{-2})$ which is optimal in terms of $\varepsilon$-dependence, under the stronger stochastic oracle, see  \cite{foster2019complexity}. This complexity was previously obtained only with multi-loop algorithms \cite{alacaoglu2026make,ji2026computation}.
\subsection{Assumptions}\label{subsec: asp}

Monotonicity and Lipschitz continuity of $G$ mean that, for all $\bx,\by$, we have
\begin{align*}
\langle {G(\bx)-G(\by)},{\bx-\by} \rangle\geq0,
\qquad
\|G(\bx)-G(\by)\|\leq L\|\bx-\by\|.
\end{align*}
We now collect the assumptions made on our main problem.
\begin{assumption}\label{asp: 1}
The function $r\colon\mathbb{R}^d \to \mathbb{R}\cup \{ +\infty \}$ is proper, closed, and convex.
The operator $G:\R^d\to\R^d$ is monotone and $L$-Lipschitz, and the set of solutions to \eqref{eq:inclusion} is nonempty.
\end{assumption}

We fix a deterministic initial point $\bz_0\in \dom r$. Let $\fil_t$ denote
the history before the oracle call at iteration $t$, that is, it is the $\sigma$-algebra generated by the randomness of $\xi_0, \dots, \xi_{t-1}$, where $\bz_t$ is
$\fil_t$-measurable, and we write $\Et[\cdot]=\E[\cdot\mid\fil_t]$.
The oracle returns an unbiased sample whose variance may grow with the
distance from the initial point.

\begin{assumption}\label{asp: 2}
At each iteration $t\geq0$, the sample $\gtil(\bz_t,\xi_t)$ satisfies 
\begin{align}\label{eq: bg}
\Et[\gtil(\bz_t,\xi_t)]=G(\bz_t) ~~~\text{and}~~~
\Et\|\gtil(\bz_t,\xi_t)-G(\bz_t)\|^2
\leq B^2\|\bz_t-\bz_0\|^2+\sigma^2.
\end{align}
\end{assumption}
The case $B=0$ recovers the bounded-variance assumption. Let us compare this to the assumption used in \cite{sohrabi2026accelerated} who required in the unconstrained case
\begin{align}\label{eq: goma_asp}
    \mathbb{E}\| \widetilde{G}(\bz_t, \xi_t) - G(\bz_t)\|^2 \leq \sigma^2 + c^2 \| G(\bz_t)\|^2.
\end{align}
To see that our assumption is weaker, we focus on the unconstrained case where $G(\bz^\star)=0$. Then, this assumption implies by Lipschitzness of $G$:
\begin{align*}
    \|G(\bz_t) \|^2 = \| G(\bz_t)-G(\bz^\star)\|^2 \leq L^2 \| \bz_t-\bz^\star\|^2 \leq 2L^2\left( \| \bz_t - \bz_0\|^2 + \| \bz_0 - \bz^\star\|^2 \right).
\end{align*}
That is, the variance assumption of \cite{sohrabi2026accelerated} implies Assumption \ref{asp: 2}.

The reason for focusing on Assumption \ref{asp: 2} is to support our claims for handling problems without bounded domains. In particular, let us assume that we have an affine operator $G(\bz) = A\bz-\bb$ for $A\in\mathbb{R}^{d\times d}$ and we calculate unbiased oracles as $\gtil(\bz) = dA_{:i}z_i-\bb$ where $i$ is selected uniformly at random and $A_{:i}$ is the $i$-th column. Then, Assumption \ref{asp: 2} holds whereas standard bounded variance or assumptions of the form \eqref{eq: goma_asp} may fail, unless the domain of the problem ($\dom r$ in our notation) is bounded. Since problems with affine $G$ and unbounded domains are common, such as linearly constrained convex optimization, our analysis is based on \Cref{asp: 2} to cover such problems.

\paragraph{Notation.} When clear from the context, we shorten
$\gtil(\bz_t,\xi_t)$ as $\gtil(\bz_t)$. For $\eta>0$, we define the proximal operator as
$\prox_{\eta r}(\bx)
=\operatorname*{argmin}_{\bz\in\R^d}
\left\{r(\bz)+\frac{1}{2\eta}\|\bz-\bx\|^2\right\}$.

The optimality condition for the problem appearing in the definition of the proximal operator, under our assumptions on $r$, gives the well known \emph{prox-inequality}:
\begin{equation}\label{eq:prox-optimality}
    \bz=\prox_{\eta r}(\bx)
    \quad\Longleftrightarrow\quad
    \bx-\bz\in \eta\partial r(\bz) \iff \langle \bz-\bx, \bu-\bz \rangle \geq \eta(r(\bz) - r(\bu)) ~~~\forall \bu \in\mathbb{R}^d.
\end{equation}

\section{A Single-Call Algorithm}\label{sec: alg_results}
Algorithm \ref{alg: alg1} is a simple extension of stochastic gradient descent (or stochastic forward-backward method), when Halpern anchoring is introduced \cite{halpern1967fixed}.
This algorithm appeared many times in the literature, with many different names, see for example \cite{ito2026last,neu2024dealing}.
It was analyzed by \cite{lee2021fast} with increasing mini-batch sizes. This method is a single-call algorithm \`{a} la \cite{hsieh2019convergence}, uses no mini-batch, variance reduction or inner loops.
\begin{algorithm}[h]
\caption{A single-call stochastic Halpern method}
\label{alg: alg1}
\begin{algorithmic}[1]
\REQUIRE Initial point $\bz_0\in \dom r$, $a\geq 2$, $H^2 = L^2+2B^2$. Set $\beta_t=\frac{a}{t+a}$ and $\eta_t=\frac{1}{H(t+a)^{3/4}}$.
\FOR{$t=0,1,\ldots$}
\STATE Obtain an unbiased estimate $\widetilde{G}(\bz_t)$ of $G(\bz_t)$
\STATE $\bz_{t+1}=\prox_{\eta_t r}
        \big((1-\beta_t)\bz_t+\beta_t\bz_0-\eta_t\gtil(\bz_t)\big)$
\ENDFOR
\end{algorithmic}
\end{algorithm}

We now present our main result which is an anytime convergence rate on the gradient mapping, or natural residual, norm. Then, we continue with its implications to anytime rates on gap functions. The proof of the result is given in Appendix \ref{sec: app_alg1}.
\begin{theorem}\label{th:main}
Let \Cref{asp: 1,asp: 2} hold. Then, for every $t\geq1$, \Cref{alg: alg1} satisfies
\begin{align*}
\E\|\mathcal G_{1/L}(\bz_t)\| \leq \sqrt{\E\|\mathcal G_{1/L}(\bz_t)\|^2}
= O\left(\frac{ (L+B)\|\bz_0-\bz^\star\| 
+ \sigma}{t^{1/4}} \right).
\end{align*}
\end{theorem}
We continue with a direct corollary for a rate on the restricted gap, by using \Cref{cor: resgap}.
\begin{corollary}\label{cor: resgap1}
    Let $r=\delta_Z$ for a convex, closed set $Z$.  We then have
\begin{equation*}
\mathbb{E}[\gap_\mathcal{U}(\bz_t)] = O\left(\frac{1}{t^{1/4}}\right),
\end{equation*}
for any compact set $\mathcal{U}$,
where $\gap_\mathcal{U}$ extends \eqref{eq: gap_def} by taking the maximum over a compact set $\mathcal{U}\subseteq Z$.
\end{corollary}
When we specialize to min-max problems, we can then particularize this result to a rate on the restricted primal-dual gap used in \cite{zheng2026lastiterate}. This result is given in the appendix as Corollary \ref{cor: minimaxgap} for direct comparison to this work.

Finally, we convert the convergence rate guarantees to iteration and sample complexity results.
\begin{corollary}\label{eq: complexity_cor}
    Under the same setup as Theorem \ref{th:main}, to obtain $\mathbb{E}\|\mathcal{G}_{\rho}(\bz_t)\|\leq\varepsilon$ or $\mathbb{E}\gap_{\mathcal{U}}(\bz_t)\leq\varepsilon$ on the last iterate $\bz_t$, the required total number of stochastic oracles is of the order $O(\varepsilon^{-4})$.
\end{corollary}
The proof of Corollaries \ref{cor: resgap1} and \ref{cor: minimaxgap} are given in Appendix \ref{sec: def_proofs}, and the proof of Corollary \ref{eq: complexity_cor} is omitted since it is immediate. 

We continue with the proof of Theorem \ref{th:main}.
The structure of our proof is inspired by the simple and powerful idea that first appeared in the work \cite{sohrabi2026accelerated} and then used also for the unconstrained result of \cite{zheng2026lastiterate}. The idea is to first analyze the deterministic shadow sequence given in \eqref{eq:reference-update} and then provide a pointwise bound between the deterministic sequence and the original stochastic sequence.

Instead of the gradient norm used in these works, we use the gradient mapping, which satisfies the necessary Lipschitz properties and then we show that the gradient mapping norm majorizes the gap function, to go around the difficulty of working directly with the gap function.

\subsection{Application for Learning in Games}
We now apply our anytime guarantee for monotone games \cite{cesa2006prediction}. Following the setup in \cite{cai2023doubly,ito2026last}, we have $N$ players where $i$-th player is denoted as $z_i$ and $\bz=(z^i; z^{-i})=(z^1, \dots, z^N)$.  Each player selects actions from sets $Z^i$ which are assumed to be compact, and $Z=Z_1\times \dots\times Z_N$. Each player minimizes a convex loss function $\ell_i(z^i; x^{-i})$. The operator $G$ will be defined as
\begin{equation*}
    G(\bz) = \begin{bmatrix}
        \nabla_{z_1} \ell_1(\bz) \\ \vdots \\\nabla_{z_N} \ell_N(\bz)
    \end{bmatrix}.
\end{equation*}
We assume access to noisy estimates of $G(\bz)$ such that $\widetilde{G}(\bz_t) = G(\bz_t) + \zeta_t$ where $\mathbb{E}[\zeta_t]=0$ and $\mathbb{E} \|\zeta_t\|^2 \leq B^2 \| \bz_t - \bz_0\|^2 + \sigma^2$. Even though this allows handling problems without bounded feasible sets, for this section, we assume that $Z$ is compact.

For direct comparison with \cite{ito2026last}, we define the constants
\begin{align*}
D=\max_{\bx,\by\in Z}\|\bx-\by\|,
\qquad
U=\max_{\bz\in Z}\|G(\bz)\|.
\end{align*}
Set $a=2$ in Algorithm \ref{alg: alg1} and define, for $\bu \in \dom r$, the regret
$
\Reg_T (\bu)
= \sum_{t=0}^{T-1} \langle G(\bz_t),\bz_t-\bu\rangle.$

A direct consequence of Theorem \ref{th:main}, Corollary \ref{cor: resgap1}, and \cite[Section 7]{ito2026last} (by selecting parameters as this work)  is that we obtain the convergence and regret guarantees
\begin{align}
\E\gap(\bz_t)
= O\left( \frac{D(LD+U+\sigma)}{t^{1/4}} \right)
~~~
\text{and}
~~~
\E\Reg_T(\bu) = O\left( D(LD+U+\sigma)T^{3/4} \right). \label{eq: comp}
\end{align}
\begin{remark}
Under the same bounded-domain and bounded-variance setting, \cite[Theorem~5.1]{ito2026last} obtained the gap bound $O(D(LD+U+\sigma)t^{-1/5})$. Thus our
result improves the anytime rate from $O(t^{-1/5})$ to $O(t^{-1/4})$. Our analysis also improves the regret guarantee obtained for the anytime algorithm in \cite{ito2026last} which was $O(T^{4/5})$ to $O(T^{3/4})$. Of course, the regret guarantee is suboptimal, however, the aim is to prove that Algorithm \ref{alg: alg1} has sublinear regret, while improving the last iterate anytime guarantee for approaching a Nash equilibrium.
\end{remark}

We now continue with the details of the analysis, following the high-level description in Section \ref{sec: alg_results}.
\subsection{Convergence of the Deterministic Sequence}

Starting from $\bar{\bz}_0:=\bz_0$, define the deterministic reference iterates by
\begin{equation}\label{eq:reference-update}
    \bar{\bz}_{t+1}
    =\prox_{\eta_t r}\big(
        (1-\beta_t)\bar{\bz}_t+\beta_t\bz_0
        -\eta_tG(\bar{\bz}_t)\big).
\end{equation}
These iterates are used only in the analysis. This sequence is indeed the algorithm analyzed by \cite{cai2026lastiterate}, referred to as the anchored gradient descent. We follow the arguments of
\cite{cai2026lastiterate}, by adapting the parameter choices. Their $t^{-1/2}$ stepsize is replaced here by
$t^{-3/4}$; the estimates below require only $H\geq L$.

The first result proves the uniform boundedness of this deterministic sequence. The proof is deferred to Appendix \ref{subsec: lem31}.
\begin{lemma}\label{lem:boundedness}
Under \Cref{asp: 1}, for the iterates generated by \eqref{eq:reference-update}, we have for any $t\geq 0$ that
\begin{equation}\label{eq:reference-bounds}
    \|\bar{\bz}_t-\bz^\star\|
    \leq\frac32\|\bz_0-\bz^\star\|,
    \qquad
    \|\bar{\bz}_t-\bz_0\|\leq \frac52 \|\bz_0-\bz^\star\|.
\end{equation}
\end{lemma}

Let us remark that using the ideas from \cite{cai2026lastiterate}, we get a guarantee on the optimality measure $\|G(\bz)+\bs\|$ where $\bs\in\partial r(\bz)$. However, we later use this bound to transfer the guarantee to the gradient mapping norm since the measure above does not admit the Lipschitzness properties required for the analysis. 

This is definitely an \emph{unsurprising} result since it is now classical. Indeed, starting from the work of \cite{yoon2021accelerated} which is extended to the constrained case by \cite{kovalev2022first,cai2022accelerated,cai2022acceleratedb}, it is known how to prove that the residual has the last iterate rate of $O(1/k)$ (with more sophisticated algorithms). In our case, since we are limited by the pointwise bounds in Section \ref{sec: pointwise}, we seek a simple bound for this sequence that will be \emph{just} sufficient for our purposes with the parameter choices for $\beta_t, \eta_t$. The proof, following \cite{cai2026lastiterate} is given in Appendix \ref{subsec: lem32}.
\begin{lemma}
\label{lem:increments}
Let \Cref{asp: 1} hold. For every $t\geq0$, we have
\begin{equation}\label{eq:increment-bound}
    \|\bar{\bz}_{t+1}-\bar{\bz}_t\|
    \leq\frac{5a\|\bz_0-\bz^\star\|}{4(t+a)}.
\end{equation}
Moreover, for every $t\geq1$, we find
$\bs_t\in\partial r(\bar{\bz}_t)$ such that
\begin{equation}\label{eq:tangent-residual-bound}
\|G(\bar{\bz}_t)+\bs_t\|
\leq\frac{5aH\|\bz_0-\bz^\star\|}{2(t-1+a)^{1/4}}
\left(\frac32+\frac{L}{2H(t-1+a)^{3/4}}\right).
\end{equation}
\end{lemma}

\subsection{Pointwise Bounds between Stochastic and Deterministic Sequences}\label{sec: pointwise}
We follow the stochastic comparison argument of \cite{sohrabi2026accelerated}, bounding the distance between the stochastic and deterministic iterates. For the constrained problem, we adapt the argument using Lipschitzness of the gradient mapping (which is due to the nonexpansiveness of the proximal operator), under \Cref{asp: 2}. Since this variance assumption is weaker than the assumption in \cite{sohrabi2026accelerated} (see Section \ref{subsec: asp}), our result extends this work even in the unconstrained case. 

Interestingly, a similar analysis was used in \cite[Section 5.4]{zheng2026lastiterate} in the unconstrained case, but not in the constrained case where the strategy was characterizing the solution of a perturbed problem, similar to \cite{ito2026last}. Both works got the $O(t^{-1/5})$ anytime rate for the gap.

Our insight for the constrained extension departs from \cite[Section 7]{sohrabi2026accelerated} and \cite[Remark 4.2]{zheng2026lastiterate} in that we do not directly analyze the gap function in the constrained setting, but go though the gradient mapping. The advantages are three-fold: (1) gradient mapping norm is Lipschitz; (2) gradient mapping norm is a classical optimality measure for problems without bounded constraint sets, making our result valid for problems with unbounded constraint sets; (3) since gradient mapping norm majorizes the (restricted) gap, the anytime rate we prove for the gradient mapping norm directly transfers to the same rate on the gap. Proof is given in Appendix \ref{subsec: lem33}.
\begin{lemma}\label{lem:tracking}
Under \Cref{asp: 1,asp: 2}, for every $t\geq0$, we have
\begin{equation}\label{eq:tracking-rate}
\E\|\bz_t-\bar{\bz}_t\|^2
\leq \frac{\sigma^2+\frac{25}{2}B^2\|\bz_0-\bz^\star\|^2}{H^2(a-\frac12)\sqrt{t+a}} \text{~and~} \E\|\bz_t-\bz_0\|^2
\leq
\frac{25}{2}\|\bz_0-\bz^\star\|^2
+2\mathbb{E}\|\bar\bz_t - \bz_t\|^2.
\end{equation}
\end{lemma}

\section{A Halpern Method with STORM Variance Reduction}\label{sec: storm}
In this section, we depart from the game setting which came with specific requirements on access to the oracle. In particular, in the previous section, it was essential that the algorithm only accessed one unbiased oracle per random seed $\xi$, that is $\gtil(z_t, \xi_t)$. However, in optimization setting, such as for linearly constrained optimization, it is possible to access two-point oracles $(\gtil(z_t, \xi_t), \gtil(z_{t-1}, \xi_t))$. As we now prove, this oracle will helps us obtain a better, in fact $\varepsilon$-optimal, complexity.

\Cref{alg: alg2} is simple combination of \Cref{alg: alg1} and the STORM estimator \cite{cutkosky2019momentum}. STORM has previously been combined with Halpern anchoring for extragradient methods \cite{alacaoglu2025towards} and forward-backward-forward methods \cite{alacaoglu2026solving}, yet only suboptimal complexity results with $\widetilde{O}(\varepsilon^{-4})$ were shown for a randomly selected iterate. 

We replace the single-call stochastic gradient in \cref{alg: alg1} with the STORM estimator, without an additional extragradient correction. This method uses two stochastic-calls and one proximal update, with no mini-batch or inner loops. We will prove the $\varepsilon$-optimal complexity of $O(\varepsilon^{-2})$ for the last iterate. Because we use variance reduction, we will use the mean-square smoothness property for the operator. This assumption implies Assumption \ref{asp: 2} since $\mathbb{E}\|G_{\xi}(\bx_0)\|^2$ is bounded. Moreover, for the example sketched in Section \ref{subsec: asp}, the following assumption is also satisfied.

\begin{algorithm}[h]
\caption{A STORM proximal Halpern method}
\label{alg: alg2}
\begin{algorithmic}[1]
\REQUIRE Initial point $\bz_0\in \dom r$, and $H=\max\left\{2L,\;4\sqrt{1024\Lambda^2+4+16B^2}\right\}$. Set $\alpha_t=\beta_t=\frac{1}{t+1}$ and $\eta_t=\frac{1}{H\sqrt{t+1}}$. Draw $\xi_0$, set $\bv_0 = \gtil(\bz_0, \xi_0)$ and $\bz_1 = \prox_{\eta_0 r}(\bz_0-\eta_0 \bv_0)$
\FOR{$t=1,2,\ldots$}
\STATE Draw one sample $\xi_t$ and obtain unbiased estimates $\gtil(\bz_t,\xi_t)$ and $\gtil(\bz_{t-1},\xi_t)$
\STATE $\bv_t = \gtil(\bz_t,\xi_t) +(1-\alpha_t)\big(\bv_{t-1}-\gtil(\bz_{t-1},\xi_t)\big)$
\STATE $\bz_{t+1}=\prox_{\eta_t r}
        \big((1-\beta_t)\bz_t+\beta_t\bz_0-\eta_t\bv_t\big)$
\ENDFOR
\end{algorithmic}
\end{algorithm}

\begin{assumption}\label{asp: 3}
For some $\xi$, we can obtain the multi-point oracle $\gtil(\bx,\xi)$ and $\gtil(\by,\xi)$. Let $G$ be Lipschitz in mean square-sense, that is:
\begin{align*}
\E_{\xi} \|\gtil(\bx, \xi)
-\gtil(\by, \xi)\|^2 
\leq\Lambda^2\|\bx-\by\|^2.
\end{align*}
Let also $\mathbb{E}\|\gtil(\bz_0, \xi)\|^2<+\infty$.
\end{assumption}

We now show a last-iterate guarantee for the tangent residual, which is a stronger optimality measure than the gradient mapping. In particular, \emph{tangent residual} (which we often call \emph{residual}) is defined as
\begin{align*}
    \res(\bz) = \min_{\bs\in\partial r(\bz)}\|G(\bz) + \bs\|.
\end{align*}
In fact, we prove the rate for the quantity $\|G(\bz) + \bs\|$, where $\bs \in \partial r(\bz)$ which majorizes the tangent residual, as well as the gradient mapping norm, see for example \cite[Prop. 2]{cai2022accelerated}.
\begin{theorem}\label{th: stmain}
Let \Cref{asp: 1,asp: 3} hold. Then, for every $t\geq1$, \Cref{alg: alg2} satisfies
\begin{align*}
\E[\res(\bz_t)] 
\leq \frac{27H{\max\{\sigma,4\|\bz_0-\bz^\star\|\}}}{\sqrt{t}}
= O\left(t^{-1/2} \right).
\end{align*}
\end{theorem}
Proof of this theorem appears in Appendix \ref{app: storm}.
Unlike the analysis of Algorithm \ref{alg: alg1}, we are not using the deterministic-tracking argument of \cite{sohrabi2026accelerated}. Instead, the proof will follow from induction, inspired by the proof in \cite{cai2026lastiterate}. The main distinction is that variance reduction brings a new recursion that is coupled with other recursions. In particular, we obtain recursions depending on the sequences $
    \mathbb{E}\|\bz_t-\bz_{t+1}\|^2, \mathbb{E}\|\bz_t - \bz^\star\|^2, \E \| \bv_t-G(\bz_t)\|^2$. After showing all these sequences have a certain rate, we can convert the result to the claimed rate on the residual.
\subsection{Three Lemmas}
We start with a recursion showing the evolution of the difference of iterates of the method. The explicit form is rather involved, however, we include the orders of the terms in the statement below.
\begin{lemma}[See \Cref{lem: stdis_app} for the full expression]\label{lem: stdis}
Under \Cref{asp: 1,asp: 3}, for every $t\geq1$ and for some $a>1$, we have
\begin{align*}
\E\|\bz_{t+1}-\bz_t\|^2
&\leq \Theta\left(1-\frac{a}{t+1}\right)
\E\|\bz_t-\bz_{t-1}\|^2\\
&\quad+ \Theta\left(\frac{1}{t^2}\right)\E \| \bv_{t-1}-G(\bz_{t-1})\|^2 + \Theta\left(\frac{1}{t^3} \right) (\E\|\bz_{t-1}-\bz_0\|^2+\sigma^2).
\end{align*}
\end{lemma}
On a high level, one can see that if $\mathbb{E}\|\bz_{t-1}-\bz_0\|^2$ was constant and if $\mathbb{E}\|\bv_{t-1} - G(\bz_{t-1})\|^2$ had a rate $O(1/t)$, then inducting on this inequality would give $\mathbb{E}\|\bz_t-\bz_{t-1}\|^2 = O(1/t^2)$. Indeed, we do not have boundedness of $\mathbb{E}\|\bz_{t-1}-\bz_0\|^2$ or the rate of decrease of $\mathbb{E}\|\bv_{t-1} - G(\bz_{t-1})\|^2$, we will  characterize the evolution of these sequences and then use induction.

We follow with the recursion for the variance reduced estimator $\bv_t$ with its proof in App. \ref{app: lem_storm1}.
\begin{lemma}\label{lem: esterr}
Under \Cref{asp: 1} and \ref{asp: 3}, for every $t\geq1$, the estimator in \cref{alg: alg2} satisfies
\begin{align*}
\Et\|\bv_t-G(\bz_t)\|^2
&\leq(1-\alpha_t)^2\|\bv_{t-1}-G(\bz_{t-1})\|^2 + 2\Lambda^2\|\bz_t-\bz_{t-1}\|^2 \\
&\quad+ 2\alpha_t^2\big(\sigma^2+B^2\|\bz_{t-1}-\bz_0\|^2\big).
\end{align*}
\end{lemma}
We next have a recursion on the distance $\mathbb{E}\|\bz_t-\bz^\star\|^2$ which is necessary for showing expected boundedness of the iterates $\bz_t$ even when the domain of the problem is not bounded. Its proof is given in Appendix \ref{app: proof_bdd}.
\begin{lemma}\label{lem: stbd}
Let \Cref{asp: 1} hold. For every $t\geq1$ we have that
\begin{align*}
\|\bz_{t+1}-\bz^\star\|^2
\leq \left(1-\frac{\beta_t}{2}\right)\|\bz_t-\bz^\star\|^2 
+ 2\beta_t \|\bz_0-\bz^\star\|^2+\frac{2\eta_t^2}{\beta_t}\|\bv_t-G(\bz_t)\|^2.
\end{align*}
\end{lemma}
\subsection{Convergence Analysis}
\begin{proof}[Proof sketch of Theorem \ref{th: stmain}]
We use induction on the three estimates given in Lemmas \ref{lem: stdis}, \ref{lem: esterr}, and \ref{lem: stbd}. This is carried out in \Cref{lem: idbd} and we obtain
\begin{align}\label{eq: ind_results}
    \E\|\bz_t-\bz_0\|^2 < C,~~~ \E\|\bv_t-G(\bz_t)\|^2 = O(1/k),~~~ \E\|\bz_t-\bz_{t+1}\bz\|^2 = O(1/k^2).
\end{align}
Moreover, expanding on the definition of the tangent residual, we obtain
\begin{align*}
    \E\|G(\bz_{t})+\bs_{t}\|
&\leq(L+\eta_{t-1}^{-1})\E\|\bz_{t}-\bz_{t-1}\|
+\left(\frac{\beta_{t-1}}{\eta_{t-1}}\right)\E\|\bz_{t-1}-\bz_0\|\\
&\quad+\E\|\bv_{t-1}-G(\bz_{t-1})\|,
\end{align*}
where the precise estimation appears in Appendix \ref{app: storm}.
Since $\eta_t = \Theta(1/\sqrt{t})$ and $\beta_t=\Theta(1/t)$, plugging in the bounds from \eqref{eq: ind_results} gives the assertion.
\end{proof}
\section{Preliminary Numerical Results}\label{sec: numres}
Our numerical experiments evaluate two constrained problems to complement our theoretical results. We first consider the problem of matrix games given as $\min_{\bx\in \Delta^{m}}\max_{\by\in \Delta^{n}}\bx^\top A\by$, where $\Delta^d$ denotes the $d$-dimensional simplex. For this problem, we set $m=n=50$ and set $A$ by using the random payoff game of \cite{abe2025boosting}, replacing their additive noise with row and column sampling for computing the unbiased oracles. We report the gap defined in \eqref{eq: gap_def}, which is equivalent with the primal dual gap for this problem. Next we show the robust least squares problem studied by \cite{cai2022stochastic}, with $\ell_\infty$ constraints: $\min_{\bx\in \mathcal{X}} \max_{\by\in\mathcal Y} \frac{1}{2n}\|A\bx-\by\|^2-\frac{1.5}{2n}\|\by-\bb\|^2$, where $\mathcal X=\{\bx:\|\bx\|_\infty\leq1\}$ and $\mathcal Y=\{\by:\|\by-\bb\|_\infty\leq1/2\}$. Following \cite{cai2022stochastic}, we use the same real dataset size of $21263\times81$ from \cite{hamidieh2018data,dua2017uci}. We estimate the operator by uniformly sampling a data row and report the gradient mapping norm.

We compare \Cref{alg: alg1} and \Cref{alg: alg2} with RG and ROG of \cite{ito2026last}, and anytime PS-OGDA of \cite{zheng2026lastiterate}. For each method, the hyperparameters are selected using three validations and then fixed for evaluation. The reported curves are averages over ten independent runs. We observe that \Cref{alg: alg1} performs comparably to or slightly better than the baselines, while \Cref{alg: alg2} outperforms over the baselines on both problems. 

\begin{figure*}[ht!]
\centering
\subfigure{%
\includegraphics[width=0.45\linewidth]{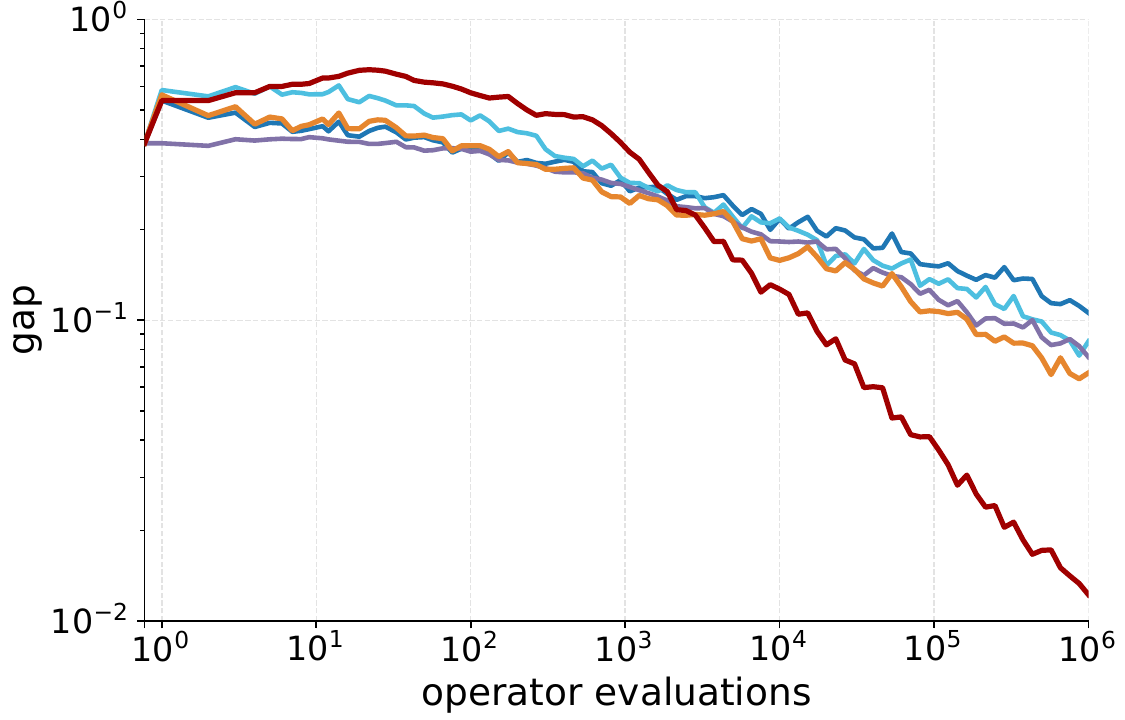}
}\hfill
\subfigure{%
\includegraphics[width=0.45\linewidth]{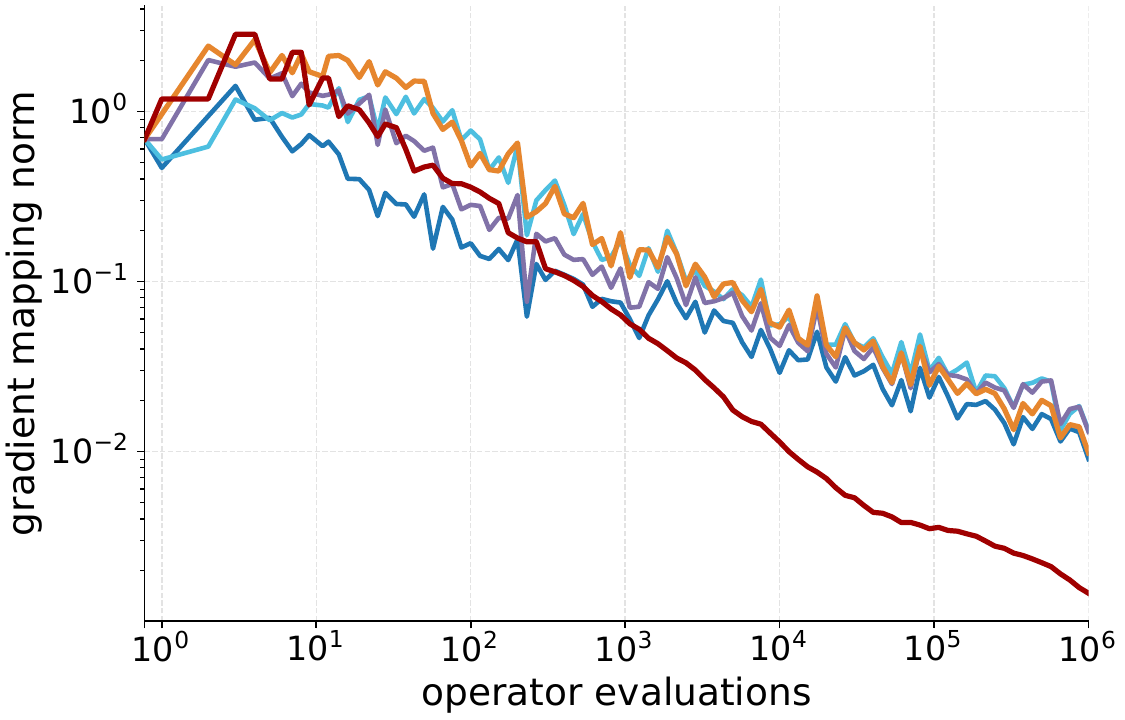}
}\hfill
\subfigure{%
\includegraphics[width=0.60\linewidth]{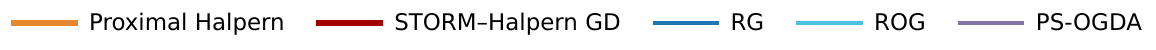}
}
\caption{\small{Left: matrix games. Right: robust least squares problem.}
} 
\label{fig: fig1}
\end{figure*}
\section{Conclusions}
We studied single-loop, anytime algorithms and proved the rates of convergence $O(t^{-1/4})$ and $O(t^{-1/2})$, where the former is by using a single call to the unbiased oracle at every iteration with no variance reduction and the latter is by using two calls, with variance reduction. Directions for future work include obtaining a $\widetilde{O}(\varepsilon^{-2})$ complexity with a single-loop method, handling nonmonotonicity, and getting optimal dependence on non-dominant terms in convergence bounds.

\section*{AI Usage Disclosure}
The authors used generative AI tools for  the straightforward but tedious numerical estimations used in Facts \ref{fact: sces} and \ref{fact: stes}, in particular to obtain precise constants. The authors then verified, simplified, and rewritten all of these estimations. After setting up the algorithm and the proof template for Section \ref{sec: storm}, and verifying the orders of the terms in the analysis for $\eta_t$ of order $1/t^{1/3}$, the authors asked these tools to check the orders of the terms when $\eta_t$ has order $1/t^{1/2}$ to check feasibility. After this, the authors wrote the proof. The authors used generative AI tools for coding for the numerical experiments.

\section*{Acknowledgments}
This research was funded by the Natural Sciences and Engineering Research Council of Canada (NSERC), [funding reference number RGPIN-2025-06634].

\bibliographystyle{alpha}
\bibliography{lit}

\newpage
\appendix

\section{Analysis of Algorithm \ref{alg: alg1}}\label{sec: app_alg1}
\begin{proof}[Proof of \Cref{th:main}]
Recall that $\rho=1/L$. By adding and subtracting $\mathcal G_\rho(\bar \bz_t)$ and applying Young's inequality, we obtain 
\begin{align}
\E\|\mathcal G_\rho(\bz_t)\|^2
&\leq 2\|\mathcal G_\rho(\bar{\bz}_t)\|^2+2 \E\|\mathcal G_\rho (\bz_t)-\mathcal G_\rho(\bar{\bz}_t)\|^2 \notag \\
&\leq 2\|\mathcal G_\rho(\bar{\bz}_t)\|^2+2(\rho^{-2}+L^2) \E\|\bz_t-\bar{\bz}_t\|^2,\label{eq: uio5}
\end{align}
where the last line is by \eqref{eq:mapping-lipschitz}.
We bound the first term on the right-hand side. By \Cref{lem:mapping-comparison,lem:increments}, we get
\begin{align*}
\|\mathcal G_\rho (\bar \bz_t)\| 
\leq \frac{5aH\|\bz_0-\bz^\star\|}{2(t-1+a)^{1/4}} \left(\frac32+\frac{L}{2H(t-1+a)^{3/4}}\right).
\end{align*}
Taking the square of both sides and multiplying by $2$ give us
\begin{align*}
2\|\mathcal G_\rho (\bar \bz_t)\|^2
&\leq \frac{25a^2H^2\|\bz_0-\bz^\star\|^2}{2\sqrt{t-1+a}} \left(\frac32+\frac{L}{2Ha^{3/4}}\right)^2,
\end{align*}
where we used $t-1+a\geq a$. 

For the second term on the right-hand side of \eqref{eq: uio5}, by \Cref{lem:tracking}, we have
\begin{align*}
2(\rho^{-2}+L^2) \E\|\bz_t-\bar{\bz}_t\|^2
\leq\frac{2(\rho^{-2}+L^2)}{H^2(a-\frac12)\sqrt{t-1+a}}
\left(\sigma^2+\frac{25}{2}B^2\|\bz_0-\bz^\star\|^2\right),
\end{align*}
where we used $t+a\geq t-1+a$.

Combining the last two estimates in \eqref{eq: uio5} yields
\begin{align*}
\E \|\mathcal G_\rho(\bz_t)\|^2
&\leq \frac{1}{\sqrt{t-1+a}}
\Bigg[
\frac{25a^2H^2\|\bz_0-\bz^\star\|^2}{8} \left(3+\frac{L}{Ha^{3/4}} \right)^2 \\
&\qquad\qquad+
\frac{2(\rho^{-2}+L^2)}{H^2(a-\frac12)} \left( \sigma^2+\frac{25}{2}B^2\|\bz_0-\bz^\star\|^2 \right) 
\Bigg].
\end{align*}
Using $H\geq L$ and $a\geq 2$, absorbing numerical constants  and Jensen's inequality, that is,
\begin{align*}
\E\|\mathcal G_\rho(\bz_t)\|
\leq \left(\E\|\mathcal G_\rho(\bz_t)\|^2\right)^{1/2},
\end{align*}
provides the assertion.
\end{proof}
\subsection{Proof of Lemma \ref{lem:boundedness}}\label{subsec: lem31}
\begin{proof}[Proof of Lemma \ref{lem:boundedness}]
By the definition of the solution, we have $-G(\bz^\star)\in\partial r(\bz^\star)$ and
$\bz^\star=\prox_{\eta_t r}(\bz^\star-\eta_tG(\bz^\star))$.
With this, nonexpansiveness of the proximal operator, and triangle inequality, we obtain 
\begin{align}
    \|\bar\bz_{t+1} - \bz^\star\| &\leq \|(1-\beta_t)\bar\bz_{t} +\beta_t\bz_0 - \bz^\star - \eta_t G(\bar\bz_t)+\eta_t G(\bz^\star)\|\notag \\
    &\leq \| (1-\beta_t)(\bar\bz_t-\bz^\star) - \eta_t (G(\bar\bz_t) - G(\bz^\star))\| + \beta_t\|\bz_0-\bz^\star\|.\label{eq: sfi4}
\end{align}
For the first norm on the right-hand side, we use Lemma \ref{lem: normsum} with $\bu=\bar\bz_t$, $\bs=\bz^\star$, $\alpha=(1-\beta_t)$, $\gamma=\eta_t$
to get
\begin{equation}\label{eq:boundedness-recursion}
    \|\bar{\bz}_{t+1}-\bz^\star\|
    \leq
    \sqrt{(1-\beta_t)^2+\eta_t^2L^2}\,
        \|\bar{\bz}_t-\bz^\star\|
    +\beta_t\|\bz_0-\bz^\star\|.
\end{equation}
We first verify the induction base case $t=1$. At the first step, $\bar{\bz}_0=\bz_0$, $\beta_0=1$ and $\eta_0=1/(Ha^{3/4})$. 
By definition of $\bar\bz_1$ and Lemma \ref{lem: normsum} with $\bu=\bz_0$, $\bs=\bz^\star$, $\alpha=1$, $\gamma=\eta_0$, we have
\begin{align*}
\|\bar{\bz}_1-\bz^\star\|^2 \leq \| \bz_0-\bz^\star - \eta_0(G(\bz_0)-G(\bz^\star))\|^2
\leq(1+\eta_0^2L^2)\|\bz_0-\bz^\star\|^2
\leq\frac{9}{4}\|\bz_0-\bz^\star\|^2.
\end{align*}
Indeed, $\eta_0^2L^2=L^2/(H^2a^{3/2})\leq a^{-3/2}\leq 5/4$ since $H\geq L$, and $a\geq2$.

We now prove the first bound in \eqref{eq:reference-bounds} by induction. The case $t=0$ follows from $\bar{\bz}_0=\bz_0$, and the calculation above proves the case $t=1$. Suppose that, for some $t\geq1$, $\|\bar{\bz}_t-\bz^\star\| \leq\frac32\|\bz_0-\bz^\star\|$. For $t\geq1$, \Cref{fact: sces} gives $\sqrt{(1-\beta_t)^2+\eta_t^2L^2} \leq 1-\frac23\beta_t$. Using this and the inductive assumption in \eqref{eq:boundedness-recursion}, we obtain
\begin{align*}
\|\bar{\bz}_{t+1}-\bz^\star\|
&\leq \left(1-\frac23\beta_t\right)\|\bar{\bz}_t-\bz^\star\| + \beta_t\|\bz_0-\bz^\star\|\notag\\
&\leq \left(1-\frac23\beta_t\right) \frac{3}{2}\|\bz_0-\bz^\star\| +\beta_t \|\bz_0 - \bz^\star\|\notag\\
&=\frac32\|\bz_0-\bz^\star\|. 
\end{align*}
This completes the induction and proves the first bound in \eqref{eq:reference-bounds}. Finally, using the triangle inequality and the first bound in \eqref{eq:reference-bounds} gives
the second bound in \eqref{eq:reference-bounds}.
\end{proof}

\subsection{Proof of Lemma \ref{lem:increments}}\label{subsec: lem32}
\begin{proof}[Proof of \Cref{lem:increments}]
This proof follows the structure of \cite{cai2026lastiterate} only by changing the parameter schedules.

By the definition  in \eqref{eq:reference-update} and the definition of the proximal operator gives
\begin{equation}\label{eq:proximal-subgradient}
    \bs_{t+1}
    :=\frac{(1-\beta_t)\bar{\bz}_t+\beta_t\bz_0
             -\eta_tG(\bar{\bz}_t)-\bar{\bz}_{t+1}}{\eta_t}
    \in\partial r(\bar{\bz}_{t+1}).
\end{equation}
In particular,
$\bar{\bz}_{t+1}
=\prox_{\eta_{t+1}r}(\bar{\bz}_{t+1}+\eta_{t+1}\bs_{t+1}) \iff \bar\bz_{t+1}+\eta_{t+1}\partial r(\bar\bz_{t+1})\ni \bar\bz_{t+1} + \eta_{t+1}\bs_{t+1}$. We compare this expression with the update  $\bar{\bz}_{t+2} = \prox_{\eta_{t+1}r}((1-\beta_{t+1})\bar\bz_{t+1} +\beta_{t+1}\bz_0- \eta_{t+1}G(\bar\bz_{t+1}))$. By \Cref{fact: sces}, we have $\eta_{t+1}/\eta_t-\beta_{t+1}\geq0$. Nonexpansiveness of the proximal operator and substitution of \eqref{eq:proximal-subgradient} give the first inequality below. Collecting the coefficients of $\bar{\bz}_{t+1}$, $\bar{\bz}_t$, and $\bz_0$ gives the equality, while the last inequality follows from the triangle inequality.
\begin{align}
&\|\bar \bz_{t+2}-\bar \bz_{t+1}\|\notag \\
&\leq \Bigg \|-\beta_{t+1}\bar \bz_{t+1}+\beta_{t+1}\bz_0-\eta_{t+1}G(\bar \bz_{t+1}) -\frac{\eta_{t+1}}{\eta_t}\big((1-\beta_t)\bar \bz_{t}+\beta_t\bar \bz_{0}-\eta_tG(\bar \bz_{t}) - \bar\bz_{t+1}\big) \Bigg\| \notag\\
&= \Bigg\| \left(\frac{\eta_{t+1}}{\eta_t}-\beta_{t+1}\right)(\bar \bz_{t+1}-\bar \bz_{t})-\eta_{t+1}\big(G(\bar \bz_{t+1})-G(\bar \bz_{t})\big) +\left(\beta_{t+1}-\frac{\eta_{t+1}}{\eta_t}\beta_t\right)(\bar \bz_{0}-\bar \bz_{t})\Bigg\| \notag\\
&\leq \Bigg\| \left(\frac{\eta_{t+1}}{\eta_t}-\beta_{t+1}\right)(\bar \bz_{t+1}-\bar \bz_{t})-\eta_{t+1}\big(G(\bar \bz_{t+1})-G(\bar \bz_{t})\big)\Bigg\| +\left|\beta_{t+1}-\frac{\eta_{t+1}}{\eta_t}\beta_t\right|\|\bar \bz_{0}-\bar \bz_{t}\|. \label{eq: referdisplacement}
\end{align}
We next estimate the first norm on the right-hand side. Using Lemma \ref{lem: normsum} with $\bu=\bar\bz_{t+1}$, $\bs=\bar\bz_t$, $\alpha = \frac{\eta_{t+1}}{\eta_t} - \beta_{t+1}$, $\gamma=\eta_{t+1}$ yields
\begin{align*}
&\Bigg\| \left(\frac{\eta_{t+1}}{\eta_t}-\beta_{t+1}\right)(\bar \bz_{t+1}-\bar \bz_{t})-\eta_{t+1}\big(G(\bar \bz_{t+1})-G(\bar \bz_{t})\big)\Bigg\|^2\\
&\leq \left[\left(\frac{\eta_{t+1}}{\eta_t}-\beta_{t+1}\right)^2 +\eta_{t+1}^2L^2 \right]
\|\bar{\bz}_{t+1}-\bar{\bz}_t\|^2.
\end{align*}
Taking square roots, substituting this estimate into \eqref{eq: referdisplacement}, and using  \Cref{lem:boundedness} gives
\begin{align}\label{eq:increment-general-recursion}
\|\bar{\bz}_{t+2}-\bar{\bz}_{t+1}\|
&\leq
\sqrt{\left(\frac{\eta_{t+1}}{\eta_t}-\beta_{t+1}\right)^2+\eta_{t+1}^2L^2}\,
\|\bar{\bz}_{t+1}-\bar{\bz}_t\|
\notag \\
&\quad+ 
\frac52\left|\beta_{t+1}-\frac{\eta_{t+1}}{\eta_t}\beta_t\right|\|\bz_0-\bz^\star\|.
\end{align}
By \cref{fact: sces}, we have
\begin{align*}
\left(\frac{\eta_{t+1}}{\eta_t}-\beta_{t+1}\right)^2 +\eta_{t+1}^2L^2
\leq \left(1-\frac{3}{2s}\right)^2,
~~~
\text{and}
~~~
\left|\beta_{t+1} -\frac{\eta_{t+1}}{\eta_t}\beta_t\right|
\leq \frac{a}{4(t+a)(t+a+1)}.
\end{align*}
Substituting these bounds into \eqref{eq:increment-general-recursion} gives
\begin{equation}\label{eq:increment-recursion}
\|\bar{\bz}_{t+2}-\bar{\bz}_{t+1}\|
\leq\left(1-\frac{3}{2(t+a+1)}\right) \|\bar{\bz}_{t+1}-\bar{\bz}_t\|
+\frac{5a\|\bz_0-\bz^\star\|}{8(t+a)(t+a+1)}.
\end{equation}
For $t=0$, since $\bar\bz_1 = \prox_{\eta_0 r}(\bz_0 - \eta_0 G(\bz_0))$, the definition of the proximal operator gives
\begin{align*}
    \langle \bar\bz_1 - \bz_0 + \eta_0 G(\bz_0), \bz^\star-\bar\bz_1 \rangle \geq \eta_0\left( r(\bar\bz_1) - r(\bz^\star)\right)
\end{align*}
By $2\langle \bu, \bs \rangle = \|\bu+\bs\|^2 - \| \bu\|^2 - \| \bs\|^2$ and adding and subtracting $\eta_0\langle G(\bz^\star), \bz^\star-\bar\bz_1\rangle$, we get
\begin{align*}
    &\| \bz_0-\bz^\star\|^2-\|\bar\bz_1 - \bz_0\|^2 - \| \bar\bz_1-\bz^\star\|^2  + 2\eta_0\langle G(\bz_0)-G(\bz^\star), \bz^\star-\bar\bz_1\rangle \\
    &\qquad\qquad\qquad\qquad\geq 2\eta_0\left(r(\bar\bz_1) - r(\bz^\star) +\langle G(\bz^\star), \bar\bz_1-\bz^\star \rangle \right)
\end{align*}
The right-hand side is nonnegative by \eqref{eq: vi_def}
and rearranging gives
\begin{align*}
    \|\bar{\bz}_1-\bz_0\|^2
    &\leq\|\bz_0-\bz^\star\|^2-\|\bar{\bz}_1-\bz^\star\|^2
        -2\eta_0\inprod{\bar{\bz}_1-\bz^\star}
                          {G(\bz_0)-G(\bz^\star)}\\
    &\leq\|\bz_0-\bz^\star\|^2
        +\eta_0^2\|G(\bz_0)-G(\bz^\star)\|^2
    \leq(1+a^{-3/2})\|\bz_0-\bz^\star\|^2
    \leq \frac{25}{16}\|\bz_0-\bz^\star\|^2.
\end{align*}
Where the last two inequalities use $H\geq L$ and $a\geq2$, respectively. Taking square roots in the preceding estimate gives $\|\bar \bz_{1}-\bar \bz_{0}\|\leq \frac54 \|\bz_{0}-\bz^\star\|$, which proves \eqref{eq:increment-bound} for $t=0$. Suppose that, for some $t \geq 0$,
\begin{align*}
\|\bar \bz_{t+1} -\bar \bz_{t}\|
\leq
\frac{5a\|\bz_0 - \bz^\star\|}{4(t+a)}.
\end{align*}
Substituting this bound into \eqref{eq:increment-recursion} gives
\begin{align*}
\|\bar \bz_{t+2} -\bar \bz_{t+1}\|
&\leq \left(1-\frac{3}{2(t+a+1)}\right) \frac{5a\|\bz_0-\bz^\star\|}{4(t+a)} +\frac{5a\|\bz_0-\bz^\star\|}{8(t+a)(t+a+1)} \\
&=  \left(1-\frac{3}{2(t+a+1)}+\frac{1}{2(t+a+1)}\right)\frac{5a\|\bz_0-\bz^\star\|}{4(t+a)} \\
&= \frac{5a\|\bz_0-\bz^\star\|}{4(t+a+1)}.
\end{align*}
This completes the induction and proves \eqref{eq:increment-bound}. 

Finally, setting $t\leftarrow t-1$ in \eqref{eq:proximal-subgradient} and rearranging gives
\begin{align*}
G(\bar \bz_T)+\bs_T = G(\bar \bz_T)-G(\bar \bz_{T-1})
+\frac{\beta_{T-1}(\bz_0-\bar \bz_{T-1})+\bar \bz_{T-1}-\bar \bz_{T}}{\eta_{T-1}}.
\end{align*}
Therefore, the triangle inequality and Lipschitzness gives
\begin{align*}
\|G(\bar{\bz}_t)+\bs_t\|
&\leq L\|\bar{\bz}_t-\bar{\bz}_{t-1}\|
+\frac{\beta_{t-1}\|\bz_0-\bar{\bz}_{t-1}\|
+\|\bar{\bz}_t-\bar{\bz}_{t-1}\|}{\eta_{t-1}}\\
&\leq \frac{5aL\|\bz_0-\bz^\star\|}{4(t-1+a)}+H(t-1+a)^{3/4}
\left(\frac{a}{t-1+a}\cdot\frac52\|\bz_0-\bz^\star\|+\frac{5a\|\bz_0-\bz^\star\|}{4(t-1+a)}\right) \\
&= \frac{5aL\|\bz_0-\bz^\star\|}{4(t-1+a)}
+\frac{15aH\|\bz_0-\bz^\star\|}{4(t-1+a)^{1/4}} \\
&= \frac{5aH\|\bz_0-\bz^\star\|}{2(t-1+a)^{1/4}}
\left(\frac32 +\frac{L}{2H(t-1+a)^{3/4}} \right).
\end{align*}
The second inequality uses \Cref{lem:boundedness}, \eqref{eq:increment-bound}, $\beta_{t-1}=a/(t-1+a)$, and $\eta_{t-1}^{-1}=H(t-1+a)^{3/4}$.
Since $\bs_t\in\partial r(\bar{\bz}_t)$, this proves \eqref{eq:tangent-residual-bound}.
\end{proof}

\subsection{Proof of Lemma \ref{lem:tracking}}\label{subsec: lem33}
\begin{proof}[Proof of Lemma \ref{lem:tracking}]
By the Young's inequality and \Cref{lem:boundedness}, we have
\begin{equation}\label{eq:radius-comparison}
\|\bz_t-\bz_0\|^2\leq 2\|\bz_t-\bar{\bz}_t\|^2 + 2\|\bz_0-\bar{\bz}_t\|^2
\leq2\|\bz_t-\bar{\bz}_t\|^2
+\frac{25}{2}\|\bz_0-\bz^\star\|^2.
\end{equation}
Using the definitions of $\bz_{t+1}$, $\bar\bz_{t+1}$, and the
nonexpansiveness of the proximal operator, we obtain
\begin{align*}
 \|\bz_{t+1}-\bar{\bz}_{t+1}\|^2 
 &= \|\prox_{\eta_t r}\big((1-\beta_t)\bz_t+\beta_t\bz_0 - \eta_t\gtil(\bz_t)\big) \\
 &\quad-\prox_{\eta_t r}\big((1-\beta_t)\bar{\bz}_t+\beta_t\bz_0-\eta_tG(\bar{\bz}_t)\big)\|^2 \\
&\leq \big\|(1-\beta_t)(\bz_t-\bar{\bz}_t)
-\eta_t\big(\gtil(\bz_t) -G(\bar{\bz}_t)\big)\big\|^2.
\end{align*}
We add and subtract $G(\bz_t)$ inside the norm and take conditional expectation. Since the stochastic oracle is unbiased, the noise cross term vanishes, yielding
\begin{align}
\Et\|\bz_{t+1}-\bar{\bz}_{t+1}\|^2
&\leq
\big\|(1-\beta_t)(\bz_t-\bar{\bz}_t) -\eta_t\big(G(\bz_t)-G(\bar{\bz}_t)\big)\big\|^2 
+\eta_t^2\Et\|\gtil(\bz_t)-G(\bz_t)\|^2.
\label{eq: referencing}
\end{align}
We use Lemma \ref{lem: normsum} with $\bu=\bz_t$, $\bs=\bar\bz_t$, $\alpha=1-\beta_t$, $\gamma=\eta_t$ to get
\begin{align*}
\big\|
(1-\beta_t)(\bz_t-\bar{\bz}_t) - \eta_t(G(\bz_t)-G(\bar{\bz}_t))
\big\|
\leq \sqrt{(1-\beta_t)^2+\eta_t^2L^2}\|\bz_t-\bar{\bz}_t\|.
\end{align*}
Next, \Cref{asp: 2} and \eqref{eq:radius-comparison} bound the last term
on the right-hand side as
\begin{align}\label{eq: bg_bd0}
\Et\|\gtil(\bz_t)-G(\bz_t)\|^2
\leq B^2\|\bz_t-\bz_0\|^2+\sigma^2 \leq 2B^2\|\bz_t-\bar{\bz}_t\|^2
+\frac{25}{2}B^2\|\bz_0-\bz^\star\|^2+\sigma^2.
\end{align}
Substituting these two bounds into \eqref{eq: referencing} gives
\begin{align}\label{eq: bhj5}
\Et\|\bz_{t+1}-\bar{\bz}_{t+1}\|^2
&\leq\big((1-\beta_t)^2+\eta_t^2(L^2+2B^2)\big) \|\bz_t-\bar{\bz}_t\|^2
+\eta_t^2\left( \frac{25}{2}B^2\|\bz_0-\bz^\star\|^2+\sigma^2 \right). 
\end{align}
For the rest of the proof let us set 
\begin{equation}\label{eq: v_def}
    V_\star = \frac{25}{2}B^2\|\bz_0-\bz^\star\|^2+\sigma^2.
\end{equation}
For $t\geq1$, \Cref{fact: sces} gives
\begin{align*}
(1-\beta_t)^2+\eta_t^2(L^2+2B^2)
\leq 1-\frac{a}{t+a}.
\end{align*}
Taking expectation in \eqref{eq: bhj5} and using the bound together with \eqref{eq: v_def} yields
\begin{equation}\label{eq:tracking-recursion}
\E\|\bz_{t+1}-\bar{\bz}_{t+1}\|^2
\leq\left(1-\frac{a}{t+a}\right)
\E\|\bz_t-\bar{\bz}_t\|^2
+\frac{V_\star}{H^2(t+a)^{3/2}},
\qquad t\geq1.
\end{equation}
At $t=0$, $\bz_0=\bar{\bz}_0$. Substituting this into \eqref{eq: referencing} and using $\eta_0^2=1/(H^2a^{3/2})$ gives
\begin{align*}
\E\|\bz_1-\bar{\bz}_1\|^2
\leq \eta_0^2\E\|\gtil(\bz_0)-G(\bz_0)\|^2 \leq \eta_0^2 V_\star =\frac{V_\star} {H^2a^{3/2}},
\end{align*}
where the second inequality used \eqref{eq: bg_bd0} and \eqref{eq: v_def}.
Thus, \eqref{eq:tracking-recursion} also holds at $t=0$.

We now prove \eqref{eq:tracking-rate} by induction. The case $t=0$ follows from $\bz_0=\bar{\bz}_0$. Suppose that, for some $t\geq0$,
\begin{align*}
\E\|\bz_t-\bar{\bz}_t\|^2
\leq \frac{V_\star}{H^2(a-\frac12)\sqrt{t+a}}.
\end{align*}
Since $1-\frac{a}{t+a}\geq0$, substituting the inductive assumption into
\eqref{eq:tracking-recursion} gives
\begin{align*}
\E\|\bz_{t+1}-\bar{\bz}_{t+1}\|^2
&\leq \left(1-\frac{a}{t+a}\right) \frac{V_\star}{H^2(a-\frac12)\sqrt{t+a}} +\frac{V_\star}{H^2(t+a)^{3/2}} \\
&= \frac{V_\star}{H^2(a-\frac12)}
\left( \frac{1}{\sqrt{t+a}}-\frac{1}{2(t+a)^{3/2}} \right) \\
&\leq \frac{V_\star}{H^2(a-\frac12)\sqrt{t+a+1}}.
\end{align*}
The last inequality follows from convexity of $s\mapsto s^{-1/2}$ for $s=t+a$, which gives $(s+1)^{-1/2}\geq s^{-1/2}-\frac12 s^{-3/2}$. This completes the induction and proves \eqref{eq:tracking-rate} after using \eqref{eq: v_def}.

Finally, taking expectations in \eqref{eq:radius-comparison} and using \eqref{eq:tracking-rate} gives
\begin{align*}
\E\|\bz_t-\bz_0\|^2
\leq 2\E\|\bz_t - \bar \bz_t\|^2
+\frac{25}{2}\|\bz_0-\bz^\star\|^2 \leq \frac{25}{2}\|\bz_0-\bz^\star\|^2
+ \frac{2V_\star}{H^2(a-\frac12)\sqrt{t+a}},
\end{align*}
which gives \eqref{eq:tracking-rate} after using \eqref{eq: v_def}.
\end{proof}

\newpage
\section{Analysis for Algorithm \ref{alg: alg2}}\label{app: storm}
\begin{proof}[Proof of \Cref{th: stmain}]
By the update in \Cref{alg: alg2} and the prox-inequality in \eqref{eq:prox-optimality}, for every $t\geq 1$, we have 
\begin{align}\label{eq: stprop}
\bs_{t} =\frac{(1-\beta_{t-1})\bz_{t-1}+\beta_{t-1}\bz_0-\bz_{t}}{\eta_{t-1}} -\bv_{t-1}
\in \partial r(\bz_{t}).
\end{align}
The tangent residual satisfies
\begin{align*}
\dist \big(0,G(\bz_{t})+\partial r(\bz_{t})\big) 
& \leq \|G(z_{t})+\bs_{t}\| \\
&= \Big\|G(\bz_{t}) +\frac{(1-\beta_{t-1})\bz_{t-1}+\beta_{t-1}\bz_0-\bz_{t}}{\eta_{t-1}} -\bv_{t-1} \Big\|\\
&\leq \|G(\bz_{t})-G(\bz_{t-1})\|+\frac{1}{\eta_{t-1}}\|\bz_{t-1}-\bz_{t}\|
+\frac{\beta_{t-1}}{\eta_{t-1}}\|\bz_0-\bz_{t-1}\| \\
&\quad+\|\bv_{t-1}-G(\bz_{t-1})\|\\
&\leq(L+\eta_{t-1}^{-1})\|\bz_{t}-\bz_{t-1}\|
+\frac{\beta_{t-1}}{\eta_{t-1}}\|\bz_{t-1}-\bz_0\|
+\|\bv_{t-1}-G(\bz_{t-1})\|.
\end{align*}
The equality follows from \eqref{eq: stprop}. The second and last inequalities follow from the triangle inequality and Lipschitzness, respectively. 
Taking expectations and applying Cauchy-Schwarz gives
\begin{align*}
\E[\res(\bz_t)]
&\leq (L+\eta_{t-1}^{-1}) \sqrt{\E\|\bz_t- \bz_{t-1}\|^2}\\
&\quad+ \frac{\beta_{t-1}}{\eta_{t-1}} \sqrt{\E\|\bz_{t-1}- \bz_0\|^2}
+ \sqrt{\E\|\bv_{t-1}- G(\bz_{t-1})\|^2}.
\end{align*}
By the parameter definitions, we have $\eta_{t-1}=1/(H\sqrt{t})$ and $\beta_{t-1}/\eta_{t-1}=H/\sqrt t$. Substituting \eqref{eq: stid1} and \eqref{eq: stid2} from \Cref{lem: idbd} therefore give us
\begin{align*}
\E[\res(\bz_t)]
&\leq \frac{16(L+H\sqrt{t}) \sqrt{C}}{t+1}
+\frac{2H\sqrt{C}}{\sqrt{t}}
+\frac{\sqrt{(1024\Lambda^2+4+16B^2)C}}{\sqrt{t}} \\
&\leq \frac{\big(16L+18H+\sqrt{1024\Lambda^2+4+16B^2} \big)\sqrt{C}}{\sqrt{t}} \\
&\leq \frac{27H\sqrt{C}}{\sqrt{t}}.
\end{align*}
The second inequality uses $t+1\geq t\geq\sqrt{t}$ for $t\geq1$. The last inequality follows the definition of $H$, which gives $L\leq H/2$ and $\sqrt{1024\Lambda^2+4+16B^2}\leq H/4$. Thus, we have $16L+18H+\sqrt{1024\Lambda^2+4+16B^2} \leq 27H$.
\end{proof}

We now give the details of the analysis of \cref{alg: alg2}. It remains to control the three term, $\E\|\bz_t-\bz_\star\|$,  $\E\|\bz_t-\bz_{t-1}\|$ and  $\E\|\bv_t-G(\bz_t)\|$. The following lemma gives the bounds used in the proof of \cref{th: stmain}.

\begin{lemma}\label{lem: idbd}
Let Assumptions \ref{asp: 1}, \ref{asp: 2}, and \ref{asp: 3} hold. Then, for every $t\geq0$,
\begin{align}\label{eq: stid1}
\E\|\bz_t-\bz^\star\|^2\leq C, 
~~~ 
\E\|\bz_t-\bz_0\|^2\leq 4 C,
~~~
\E\|\bv_t-G(\bz_t)\|^2\leq \frac{(1024\Lambda^2+4+16B^2)C}{t+1},
\end{align}
and, for every $t\geq1$,
\begin{align}\label{eq: stid2}
\E\|\bz_t-\bz_{t-1}\|^2 \leq \frac{256C}{(t+1)^2}.
\end{align}
Where we define $C$ as $C = \max\left\{\sigma^2, 16\|\bz_0-\bz^\star\|^2\right\}$.
\end{lemma}
\begin{proof}[Proof of \Cref{lem: idbd}.] By the definition of $C$, we have $C\geq\sigma^2$, $\|\bz_0-\bz^\star\|^2\leq C/16$ and
\begin{align*}
\E\|\bv_0-G(\bz_0)\|^2 \leq \sigma^2 \leq C \leq(1024\Lambda^2+4+16B^2)C.
\end{align*}
Where the first inequality follows from \Cref{asp: 2} with definition of $\bv_0=\gtil(\bz_0)$, while the last inequality follows from $1024\Lambda^2+4+16B^2\geq4$.

We first verify the base case at $t=1$. Together with prox-inequality  of $\bz_1=\prox_{\eta_0 r}(\bz_0-\eta_0\bv_0)$ and $\bz^\star=\prox_{\eta_t r}(\bz^\star-\eta_tG(\bz^\star))$, using  nonexpansiveness of proximal operator gives,
\begin{align*}
\E\|\bz_1-\bz^\star\|^2 
&\leq \E\big\|\bz_0-\frac{1}{H}\bv_0- \bz^\star+\eta_0 G(\bz^\star)\big\|^2 \\
&\leq 2\Big\|\bz_0-\bz^\star-\frac{1}{H} \big(G(\bz_0)-G(\bz^\star)\big)\Big\|^2
+\frac2{H^2}\E\|\bv_0-G(\bz_0)\|^2\\
&\leq 2\left(1+\frac{L^2}{H^2}\right)\|\bz_0-\bz^\star\|^2
+\frac2{H^2}\E\|\bv_0-G(\bz_0)\|^2\\
&\leq 2\left(1+\frac{1}{4}\right)\frac{C}{16}+\frac{C}{8}
=\frac{9C}{32}\leq C.
\end{align*}
Here, the second inequality follows from Young's inequality, while the third inequality can obtain by \cref{lem: normsum} with $\alpha=1, \gamma=1/H$, $\bu=\bz_0$ and $\bs=z^\star$. The last inequality uses the definition of $H$, which implies $L^2/H^2\leq1/4$ and $(1024\Lambda^2+4+16B^2)/H^2\leq1/16$.

Using Young's inequality implies
\begin{align*}
\E\|\bz_1-\bz_0\|^2
&\leq2 \E\|\bz_1- \bz^\star\|^2+2\|\bz_0- \bz^\star\|^2\\
&\leq 2C+\frac{C}{8} \leq 64C =\frac{256C}{(1+1)^2}.
\end{align*}
These bounds finish the base case of induction at $t=1$.

Now suppose for some $t\geq1$, that
\begin{align}
\E\|\bz_t-\bz^\star\|^2 \leq C,
~~~
\E\|\bz_{t-1}-\bz^\star\|^2 \leq C, 
~~~
\E\|\bz_t-\bz_{t-1}\|^2\leq\frac{256C}{(t+1)^2}, \notag\\
~~~
\text{and}
~~~
\E\|\bv_{t-1}-G(\bz_{t-1})\|^2
\leq\frac{(1024\Lambda^2+4+16B^2)C}{t}. \label{eq: idsp1}
\end{align}

We first estimate the second term in \eqref{eq: stid1}. By Young's inequality, we have
\begin{align}
\E\|\bz_t-\bz_0\|^2
& \leq 2\E\|\bz_t- \bz^\star\|^2+2\|\bz_0-\bz^\star\|^2 \notag\\
& \leq 2C +\frac{C}{8}  \notag\\
&\leq 4C, \label{eq: zrid}
\end{align}
where the second inequality follows the property in \eqref{eq: idsp1} and definition of $C$. The same argument can be applied to $\E\|\bz_{t-1}-\bz_0\|^2\leq4C$.

We next show the last term in \eqref{eq: stid1}. Recall the recursion from \cref{lem: esterr},
\begin{align*}
\Et\|\bv_t-G(\bz_t)\|^2
&\leq(1-\alpha_t)^2\|\bv_{t-1}-G(\bz_{t-1})\|^2 + 2\Lambda^2\|\bz_t-\bz_{t-1}\|^2 \\
&\quad+ 2\alpha_t^2\big(\sigma^2+B^2\|\bz_{t-1}-\bz_0\|^2\big).
\end{align*}
Taking expectations and using \eqref{eq: idsp1} and \eqref{eq: zrid} gives,
\begin{align}
\E\|\bv_t- G(\bz_t)\|^2 
&\leq \left(1- \frac{1}{t+1}\right)^2\frac{(1024\Lambda^2+4+16B^2)C}{t} \notag\\
&\quad +\frac{512\Lambda^2C}{(t+1)^2} + \frac{2(C+4B^2C)}{(t+1)^2} \notag\\
&= \frac{(1024\Lambda^2+4+16B^2)C}{t+1} -\frac{(1024\Lambda^2+4+16B^2)C}{2(t+1)^2}
\label{eq: ideser}
\end{align}
The equality follows from \Cref{fact: stes}.

Next, we verify first term in \eqref{eq: stid1}. Taking expectations in \Cref{lem: stbd} give us
\begin{align}
\E\|\bz_{t+1}-\bz^\star\|^2
\leq \left(1-\frac{\beta_t}{2}\right) \E\|\bz_t-\bz^\star\|^2 
+ 2\beta_t \|\bz_0-\bz^\star\|^2+\frac{2\eta_t^2}{\beta_t}\E\|\bv_t-G(\bz_t)\|^2.
\end{align}
Using \eqref{eq: ideser} and definition of $C$ which implies $\|\bz_0-\bz^\star\|^2 \leq C/16$, and $\eta_t^2/\beta_t=1/H^2$,
\begin{align}
\E\|\bz_{t+1}-\bz^\star\|^2
&\leq \left(1-\frac{1}{2(t+1)}\right) C
+\frac{C}{8(t+1)}
+\frac{2(1024\Lambda^2+4+16B^2)C}{H^2(t+1)} \notag\\
&\leq \left(1- \frac{1}{2(t+1)}+ \frac{1}{8(t+ 1)}+ \frac{1}{8(t+1)}\right)C \notag\\
&=\left(1-\frac{1}{4(t+1)}\right)C\leq C.
\end{align}
The second inequality uses $(1024\Lambda^2+4+16B^2)/H^2 \leq 1/16$.

Finally, we estimate \eqref{eq: stid2}. Bring the recursion in \cref{lem: stdis}, we have 
\begin{align*}
\E\|\bz_{t+1}-\bz_t\|^2
&\leq\left(1+\frac{1}{4(t+1)}\right) 
\left(\left(\frac{\eta_t}{\eta_{t-1}}-\beta_t\right)^2+\eta_t^2L^2\right)
\E\|\bz_t-\bz_{t-1}\|^2\\
&\quad+ 10(t+1) \left(\frac{\eta_t}{\eta_{t-1}}\beta_{t-1}-\beta_t\right)^2 \E\|\bz_{t-1}-\bz_0\|^2\\
&\quad+ 10(t+1) \eta_t^2\alpha_t^2
\E\|\bv_{t-1} -G(\bz_{t-1})\|^2\\
&\quad+ \eta_t^2 \Big( 2\Lambda^2\E\|\bz_t-\bz_{t-1}\|^2 + 2\alpha_t^2\big(\sigma^2+B^2\E\|\bz_{t-1}-\bz_0\|^2\big) \Big).
\end{align*}
Applying the \eqref{eq: idsp1}, \eqref{eq: zrid} and definition of $\eta_t$ and $\beta_t$ with $\sigma^2 \leq C$, together with \cref{fact: stes} gives,
\begin{align*}
\E\|\bz_{t+1}- \bz_t\|^2
&\leq
\frac{256C}{(t+1)^2}
\left(1+ \frac{1}{4(t+1)} \right)
\left(\left(1- \frac{3}{2(t+1)} \right)^2+ \frac{1}{4(t+1)} \right)\\
&\quad+ \left(40 +\frac{41(1024\Lambda^2+4+16B^2)}{2H^2} \right)
\frac{C}{(t+1)^3}\\
&\leq
\frac{256C}{(t+1)^2}
\left(1+\frac{1}{4(t+1)}\right)
\left(\left(1- \frac{3}{2(t+1)} \right)^2+ \frac{1}{4(t+1)} \right)
+\frac{64C}{(t+1)^3}. 
\end{align*}
Moreover, applying \cref{fact: stes} give us,
\begin{align}
\E\|\bz_{t+1}-\bz_t\|^2
\leq\frac{256C}{(t+2)^2}. \label{eq: iddis}
\end{align}
These estimates give \eqref{eq: idsp1} with $t$ replaced by $t+1$. This completes the induction and proves \eqref{eq: stid1} and \eqref{eq: stid2}.
\end{proof}

We prove these bounds through three estimates followed by an induction.
\subsection{Proof of \Cref{lem: esterr}}\label{app: lem_storm1}
\begin{lemma}\label{lem: esterr_app}
Under \Cref{asp: 1} and \ref{asp: 3}, for every $t\geq1$, the estimator in \cref{alg: alg2} satisfies
\begin{align*}
\Et\|\bv_t-G(\bz_t)\|^2
&\leq(1-\alpha_t)^2\|\bv_{t-1}-G(\bz_{t-1})\|^2 + 2\Lambda^2\|\bz_t-\bz_{t-1}\|^2 \\
&\quad+ 2\alpha_t^2\big(\sigma^2+B^2\|\bz_{t-1}-\bz_0\|^2\big).
\end{align*}
\end{lemma}
\begin{proof}[Proof of \Cref{lem: esterr}]
Throughout the proof, we use the notations
\begin{align*}
\gtil(\bz_t) = \gtil(\bz_t, \xi_t) 
~~~\text{and}~~~
\gtil(\bz_{t-1}) = \gtil(\bz_{t-1}, \xi_t).
\end{align*}
By the estimator update, we have
\begin{align*}
\bv_t-G(\bz_t)
&= \gtil(\bz_t) +(1-\alpha_t)\big(\bv_{t-1}-\gtil(\bz_{t-1})\big)-G(\bz_t) \\
&= (1-\alpha_t)\big(\bv_{t-1}-G(\bz_{t-1})\big) +(1-\alpha_t)\big(G(\bz_{t-1})-\gtil(\bz_{t-1})\big) \\ 
&\quad+ \gtil(\bz_t)-G(\bz_t).
\end{align*}
Since the last two terms on the right-hand side is $0$ under conditional  expectation and since $\bv_{t-1}-G(\bz_{t-1})$ is measurable under the conditioning of $\mathbb{E}_t$,we have
\begin{align*}
\Et\langle \bv_{t-1}-G(\bz_{t-1})&, \gtil(\bz_t)-G(\bz_t) + (1-\alpha_t)(G(\bz_{t-1})-\gtil(\bz_{t-1}))\rangle =0.
\end{align*}
Thus, expanding the squared norm and taking expectations give
\begin{align*}
\Et\|\bv_t&-G(\bz_t) \|^2 =(1-\alpha_t)^2\|\bv_{t-1}-G(\bz_{t-1})\|^2 \notag\\
&\quad+ \Et\left\|
\gtil(\bz_t)-G(\bz_t)+(1-\alpha_t)(G(\bz_{t-1})-\gtil(\bz_{t-1})\right\|^2.
\end{align*}
Applying \Cref{lem: noiseupper} to last term on right-hand gives the assertion.
\end{proof}

\subsection{Proof of \Cref{lem: stbd}}\label{app: proof_bdd}
\begin{lemma}[Restatement of Lemma \ref{lem: stbd}]\label{lem: stbd_app}
Let \Cref{asp: 1} hold. For every $t\geq0$ we have that
\begin{align*}
\|\bz_{t+1}-\bz^\star\|^2
\leq \left(1-\frac{\beta_t}{2}\right)\|\bz_t-\bz^\star\|^2 
+ 2\beta_t \|\bz_0-\bz^\star\|^2+\frac{2\eta_t^2}{\beta_t}\|\bv_t-G(\bz_t)\|^2.
\end{align*}
\end{lemma}
\begin{proof}
By the definition of the solution, we have $-G(\bz^\star)\in\partial r(\bz^\star)$ and
$\bz^\star=\prox_{\eta_t r}(\bz^\star-\eta_tG(\bz^\star))$. Comparing with the update in \cref{alg: alg2}, nonexpansiveness of the proximal operator gives
\begin{align*}
\|\bz_{t+1}-\bz^\star\| 
&\leq \|(1-\beta_t)\bz_t+\beta_t\bz_0-\eta_t \bv_t-(\bz^\star-\eta_tG(\bz^\star))\| \\
&= \|(1-\beta_t)(\bz_t-\bz^\star)-\eta_t\big(G(\bz_t)-G(\bz^\star)\big)
+\beta_t(\bz_0-\bz^\star)-\eta_t\big(\bv_t-G(\bz_t) \big)\|.
\end{align*}
Taking the square of this inequality and using Young's inequality with parameter $\beta_t/(1-\beta_t)$ gives
\begin{align}
\|\bz_{t+1}-\bz^\star\|^2 
&\leq \frac{1}{1-\beta_t}\|(1-\beta_t)(\bz_t-\bz^\star)-\eta_t\big(G(\bz_t)-G(\bz^\star)\big)\|^2 \notag\\
&\quad +\frac{1}{\beta_t}\big\|\beta_t(\bz_0-\bz^\star)-\eta_t\big(\bv_t-G(\bz_t) \big)\big\|^2. \label{eq: bdmid}
\end{align}
We first estimate the first term on the right-hand side, using Lemma \ref{lem: normsum} with $\alpha=1-\beta_t$, $\gamma=\eta_t$, $\bu=\bz_{t+1}$ and $\bs=\bz^\star$ yield
\begin{align*}
\frac{1}{1-\beta_t}\|(1-\beta_t)(\bz_t-\bz^\star)-\eta_t\big(G(\bz_t)-G(\bz^\star)\big)\|^2 
&\leq \left(1-\beta_t+\frac{\eta_t^2L^2}{1-\beta_t} \right)\|\bz_t-\bz^\star\|^2 \\
&\leq \left(1-\frac{\beta_t}{2}\right)\|\bz_t-\bz^\star\|^2,
\end{align*}
where the second inequality follows from  \Cref{fact: stes}.

For the second term on the right-hand side of \eqref{eq: bdmid}, Young's inequality gives
\begin{align*}
\frac{1}{\beta_t}\big\|\beta_t(\bz_0-\bz^\star)-\eta_t\big(\bv_t-G(\bz_t) \big)\big\|^2 
\leq 2\beta_t \|\bz_0-\bz^\star\|^2+\frac{2\eta_t^2}{\beta_t}\|\bv_t-G(\bz_t)\|^2.
\end{align*}
Combining these two estimates with \eqref{eq: bdmid} proves the claim.
\end{proof}

\subsection{Full Expression of Lemma \ref{lem: stdis} and Its Proof}
\begin{lemma}\label{lem: stdis_app}
Under \Cref{asp: 1,asp: 3}, for every $t\geq1$, we have
\begin{align*}
\E\|\bz_{t+1}-\bz_t\|^2
&\leq\left(1+\frac{1}{4(t+1)}\right) 
\left(\left(\frac{\eta_t}{\eta_{t-1}}-\beta_t\right)^2+\eta_t^2L^2\right)
\E\|\bz_t-\bz_{t-1}\|^2\\
&\quad+ 10(t+1) \left(\frac{\eta_t}{\eta_{t-1}}\beta_{t-1}-\beta_t\right)^2 \E\|\bz_{t-1}-\bz_0\|^2\\
&\quad+ 10(t+1) \eta_t^2\alpha_t^2
\E\|\bv_{t-1} -G(\bz_{t-1})\|^2\\
&\quad+ \eta_t^2 \Big( 2\Lambda^2\E\|\bz_t-\bz_{t-1}\|^2 + 2\alpha_t^2\big(\sigma^2+B^2\E\|\bz_{t-1}-\bz_0\|^2\big) \Big).
\end{align*}
\end{lemma}
\begin{proof}
Throughout the proof, we use the notations
\begin{align*}
\gtil(\bz_t) = \gtil(\bz_t, \xi_t) 
~~~\text{and}~~~
\gtil(\bz_{t-1}) = \gtil(\bz_{t-1}, \xi_t).
\end{align*}
By the previous update and the prox-inequality in \eqref{eq:prox-optimality}, we have
\begin{align*}
\bs_t = \frac{(1-\beta_{t-1})\bz_{t-1}+ \beta_{t-1}\bz_0- \bz_t}{\eta_{t-1}}
-\bv_{t-1} \in\partial r(\bz_t).
\end{align*}
Consequently, with stepsize $\eta_t$, we have $\bz_t=\prox_{\eta_tr}(\bz_t+\eta_t \bs_t)$.
Comparing this with the update of $\bz_{t+1}$ and using nonexansiveness of the proximal operator gives
\begin{align}\label{eq: stdisup}
\|\bz_{t+1}-\bz_t\| 
&\leq \| (1-\beta_t)\bz_t+\beta_t\bz_0-\eta_t \bv_t-(\bz_t+\eta_t \bs_t)\| \notag \\
&= \Big\|\left(\frac{\eta_t}{\eta_{t-1}}-\beta_t\right)(\bz_t-\bz_{t-1})\notag \\ 
&\quad+\left(\frac{\eta_t}{\eta_{t-1}}\beta_{t-1}-\beta_t\right)(\bz_{t-1}-\bz_0) -\eta_t(\bv_t-\bv_{t-1})\Big\|.
\end{align}
Here the last equality follows from the definition of $\bs_t$.

We first rewrite the last term inside the norm. Using the definition of estimator $\bv_t$ in Algorithm \ref{alg: alg2}, we have,
\begin{align*}
\bv_t-\bv_{t-1} 
&= \gtil(\bz_t)-\gtil(\bz_{t-1})-\alpha_t(\bv_{t-1}-\gtil(\bz_{t-1}))\\
&=\gtil(\bz_t)-\gtil(\bz_{t-1})-\alpha_t(\bv_{t-1}-G(\bz_{t-1})) +\alpha_t (\gtil(\bz_{t-1}, \xi_t) - G(\bz_{t-1})).
\end{align*}
Substituting this into \eqref{eq: stdisup} yields
\begin{align*}
\|\bz_{t+1}-\bz_t\| 
&\leq \Big\|\left(\frac{\eta_t}{\eta_{t-1}}-\beta_t\right)(\bz_t-\bz_{t-1}) -\eta_t\big(G(\bz_t)-G(\bz_{t-1})\big) \\ 
&\quad+\left(\frac{\eta_t}{\eta_{t-1}}\beta_{t-1}-\beta_t\right)(\bz_{t-1}-\bz_0) +\eta_t\alpha_t\big(\bv_{t-1}-G(\bz_{t-1})\big)\\
&\quad-\eta_t \left(\gtil(\bz_t)-G(\bz_t)+(1-\alpha_t)\big(G(\bz_{t-1})-\gtil(\bz_{t-1})\big)\right)\Big\|.
\end{align*}
Because $\Et\left[\gtil(\bz_t)-G(\bz_t)+(1-\alpha_t)\big(G(\bz_{t-1})-\gtil(\bz_{t-1})\big)\right]=0$ and because the terms in the first and second lines of the right-hand side in the display equation above are measurable under the conditioning of $\mathbb{E}_t$, we have
\begin{align*}
\Et\|\bz_{t+1}-\bz_t\|^2
&\leq  \Big\|\left(\frac{\eta_t}{\eta_{t-1}}-\beta_t\right)(\bz_t-\bz_{t-1}) -\eta_t\big(G(\bz_t)-G(\bz_{t-1})\big) \notag\\ 
&\quad+\left(\frac{\eta_t}{\eta_{t-1}}\beta_{t-1}-\beta_t\right)(\bz_{t-1}-\bz_0) +\eta_t\alpha_t\big(\bv_{t-1}-G(\bz_{t-1})\big)\Big\|^2\notag\\
&\quad+\eta_t^2 \Et\Big\|\gtil(\bz_t)-G(\bz_t)+(1-\alpha_t)\big(G(\bz_{t-1})-\gtil(\bz_{t-1})\big)\Big\|^2.
\end{align*}
We now apply Young's inequality to the first term on the right-hand side with the parameter $1/(4(t+1))$ to retain a sufficient contraction.
\begin{align}
\Et\|\bz_{t+1}-\bz_t\|^2
&\leq \left(1+\frac{1}{4(t+1)}\right)\Big\|\left(\frac{\eta_t}{\eta_{t-1}}-\beta_t\right)(\bz_t-\bz_{t-1}) -\eta_t\big(G(\bz_t)-G(\bz_{t-1})\big)\Big\|^2 \notag\\ 
&\quad+ \left(1+{4(t+1)}\right)\Big\|\left(\frac{\eta_t}{\eta_{t-1}}\beta_{t-1}-\beta_t\right)(\bz_{t-1}-\bz_0) +\eta_t\alpha_t\big(\bv_{t-1}-G(\bz_{t-1})\big)\Big\|^2\notag\\
&\quad+\eta_t^2 \Et\Big\|\gtil(\bz_t)-G(\bz_t)
+ (1-\alpha_t)\big(G(\bz_{t-1})-\gtil(\bz_{t-1})\big)\Big\|^2. \label{eq: stdisupex}
\end{align}
For the first term on the right-hand side, using \Cref{lem: normsum} with $\alpha = \eta_t/\eta_{t-1}-\beta_t\geq0$ and $\gamma=\eta_t$ gives
\begin{align}\label{eq: stdis1}
\left(1+\frac{1}{4(t+1)}\right)
&\Big\|\left(\frac{\eta_t}{\eta_{t-1}}-\beta_t\right)(\bz_t-\bz_{t-1}) -\eta_t\big(G(\bz_t)-G(\bz_{t-1})\big)\Big\|^2  \notag\\
&\leq \left(1+\frac{1}{4(t+1)}\right)
\left(\left(\frac{\eta_t}{\eta_{t-1}}-\beta_t\right)^2+\eta_t^2L^2\right)
\|\bz_t-\bz_{t-1}\|^2.
\end{align}

We next estimate second term of the right-hand side, using Young's inequality with the fact that $2(1+4(t+1))\leq 10(t+1)$ for $t\geq0$.
\begin{align}\label{eq: stdis2}
&(1+{4(t+1)})\Big\|\Big(\frac{\eta_t}{\eta_{t-1}}\beta_{t-1}-\beta_t\Big)(\bz_{t-1}-\bz_0) +\eta_t\alpha_t\big(\bv_{t-1}-G(\bz_{t-1})\big)\Big\|^2 \notag\\
&~~~\leq 10(t+1)\left(\frac{\eta_t}{\eta_{t-1}}\beta_{t-1}-\beta_t\right)^2
\|\bz_{t-1}-\bz_0\|^2 +10(t+1)\eta_t^2\alpha_t^2\|\bv_{t-1}-G(\bz_{t-1})\|^2.
\end{align}
Applying \cref{lem: noiseupper} to last term in \cref{eq: stdisupex} and substituting the \eqref{eq: stdis1} and \eqref{eq: stdis2} into \cref{eq: stdisupex}, followed by taking expectation, proves the claim.
\end{proof}

\begin{lemma}\label{lem: noiseupper}
Suppose Assumptions \ref{asp: 1} and \ref{asp: 3} hold. Then,
\begin{align*}
\Et\big\|\gtil(\bz_t)-G(\bz_t)+(1-\alpha_t)\big(G(\bz_{t-1})-\gtil(\bz_{t-1})\big)\big\|^2 &\leq 2\Lambda^2\|\bz_t-\bz_{t-1}\|^2 \\
&\quad+2\alpha_t^2\big(\sigma^2+B^2\|\bz_{t-1}-\bz_0\|^2\big).
\end{align*}
\end{lemma}
\begin{proof}
Using Young's inequality gives
\begin{align*}
\Et\big\|\gtil(\bz_t)&-G(\bz_t)+(1-\alpha_t)\big(G(\bz_{t-1})-\gtil(\bz_{t-1})\big)\big\|^2 \\
&\quad\leq 2\Et\big\|\gtil(\bz_t)-G(\bz_t)-\gtil(\bz_{t-1})+G(\bz_{t-1})\big\|^2 +2\alpha_t^2\Et\|\gtil(\bz_{t-1})-G(\bz_{t-1})\|^2 \\
&\quad\leq 2\Lambda^2\|\bz_t-\bz_{t-1}\|^2 +2\alpha_t^2\big(\sigma^2+B^2\|\bz_{t-1}-\bz_0\|^2\big),
\end{align*}
where the last inequality follows from \Cref{asp: 1} and \Cref{asp: 3}. 
\end{proof}

\newpage
\section{Deferred Proofs for Corollaries}\label{sec: def_proofs}

\begin{proof}[Proof of \Cref{cor: resgap1}]
Setting $\bz=\bz_t$ in \eqref{eq: gapproj}, taking expectations, and applying Cauchy--Schwarz give
\begin{align}
\E\gap(\bz_t) 
\leq& \left[ 
\max_{\bu\in\mathcal B}\|\bz_0-\bu\|
+\rho\|G(\bz_0)\| 
\right]\mathbb{E}\|\mathcal{G}_\rho(\bz_t)\| + 
(1+\rho L) \sqrt{\E\|\bz_t-\bz_0\|^2} \sqrt{\E\|\mathcal G_\rho(\bz_t)\|^2} \notag\\
&\quad+\rho\E\|\mathcal G_\rho(\bz_t)\|^2.
\label{eq: gapupper}
\end{align}
By \eqref{eq:tracking-rate}, because $t\geq1$, we have $1/(\sqrt{t+a}) \leq 1/(\sqrt{a+1})$, 
\begin{align*}
\E\|\bz_t-\bz_0\|^2
\leq \frac{25}{2}\|\bz_0-\bz^\star\|^2+
\frac{2\sigma^2+25B^2\|\bz_0-\bz^\star\|^2}{H^2(a-\frac12)\sqrt a}
<\infty.
\end{align*}
Combining this bound with Theorem \ref{th:main} gives the assertion.
\end{proof}

\begin{corollary}\label{cor: minimaxgap}
Under the setup of \Cref{cor: resgap1}, suppose that $\bz=(\bx,\by)$, $\mathcal U=\mathcal U_{\bx}\times\mathcal U_{\by}$, and
\begin{align*}
G(\bz) =
\begin{pmatrix}
\nabla_{\bx}f(\bx,\by)\\
-\nabla_{\by}f(\bx,\by)
\end{pmatrix},
\end{align*}
where $f$ is convex in $\bx$ and concave in $\by$. Then, we have
\begin{align*}
\E\left[
\max_{\by'\in\mathcal U_{\by}}f(\bx_t,\by')
- \min_{\bx'\in\mathcal U_{\bx}}f(\bx',\by_t) \right]
= O\left(\frac{1}{t^{1/4}}\right).
\end{align*}
\end{corollary}

\begin{proof}[Proof of \Cref{cor: minimaxgap}]
For every $(\bx',\by')\in\mathcal U_{\bx}\times\mathcal U_{\by}$, convexity in $\bx$ and concavity in $\by$ give
\begin{align*}
f(\bx,\by')-f(\bx',\by)
&= f(\bx,\by')-f(\bx,\by) + f(\bx,\by)-f(\bx',\by)\\
&\leq \langle\nabla_{\by}f(\bx,\by),\by'-\by\rangle + \langle\nabla_{\bx}f(\bx,\by),\bx-\bx'\rangle\\
&= \left\langle G(\bz), \bz-
\begin{pmatrix}
\bx'\\
\by'
\end{pmatrix}
\right\rangle.
\end{align*}
Taking the maximum over $(\bx',\by')\in\mathcal U_{\bx}\times\mathcal U_{\by}$ gives
\begin{align*}
\max_{\by'\in\mathcal U_{\by}}f(\bx,\by') - \min_{\bx'\in\mathcal U_{\bx}}f(\bx',\by)
\leq
\gap(\bz).
\end{align*}
Taking expectations at $\bz=\bz_t$ and applying
\Cref{cor: resgap1} proves the result.
\end{proof}

\section{Technical Results}

The following lemma is generalizing a useful idea from \cite{cai2026lastiterate} and is used many times in the proofs.
\begin{lemma}\label{lem: normsum}
    Let $G\colon\mathbb{R}^d\to\mathbb{R}^d$ be monotone and $L$-Lipschitz, and $\alpha, \gamma\geq0$. Then, for any $\bu, \bs\in\mathbb{R}^d$ we have
    \begin{align*}
    \| \alpha(\bu-\bs) - \gamma(G(\bu)-G(\bs)) \| \leq \sqrt{\alpha^2+\gamma^2L^2}\|\bu-\bs\|.  
    \end{align*}
\end{lemma}
\begin{proof}
    By expanding the square of the left-hand side, we obtain
    \begin{equation}\label{eq: est_q1}
    \begin{aligned}
        \| \alpha(\bu-\bs) - \gamma(G(\bu)-G(\bs)) \|^2 &= \alpha^2 \| \bu-\bs\|^2 + \gamma^2 \| G(\bu)- G(\bs)\|^2 \\
        &\quad-2\alpha\gamma\langle \bu-\bs, G(\bu)-G(\bs) \rangle.
    \end{aligned}
    \end{equation}
    Due to Lipschitzness of $G$, we have
    \begin{align*}
        \gamma^2 \| G(\bu)-G(\bs)\|^2 \leq \gamma^2L^2\|\bu-\bs\|^2.
    \end{align*}
    By monotonicity of $G$, we have
    \begin{align*}
        -2\alpha\gamma\langle\bu-\bs, G(\bu)-G(\bs)\rangle\leq 0.
    \end{align*}
    Combining the last two estimates in \eqref{eq: est_q1} and taking the square root of both sides gives the assertion.
\end{proof}

\subsection{Results on Optimality Measures}
The following standard result proves that the rate of convergence that we prove for the gradient mapping norm directly translates to the same rate on the restricted gap, up to an additional multiplicative constant. 
\begin{corollary}\label{cor: resgap}
Under \Cref{asp: 1}, let $\mathcal{U}\subseteq Z$ be a compact set.
Let $r=\delta_Z$, where $Z$ is nonempty, closed, and convex, but not necessarily bounded. Then, for every $\bz \in Z$ and any $\rho>0$, we have
\begin{align}
\gap_\mathcal{U}(\bz) \leq
\left( 
\max_{\bu\in\mathcal B}\|\bz_0-\bu\|
+\rho\|G(\bz_0)\|
+(1+\rho L)\|\bz-\bz_0\|
\right)
\|\mathcal G_\rho(\bz)\|
+\rho\|\mathcal G_\rho(\bz)\|^2. \label{eq: gapproj}
\end{align}

\end{corollary}
\begin{proof}
Fix $\bz \in C$ and let $\bs = \proj_Z(\bz-\rho G(\bz))$. By \eqref{eq:prox-optimality}, we have $\frac{\bz-\bs-\rho G(\bz)}{\rho} \in \partial \delta_Z(\bs)$, and since $\mathcal{G}_\rho(\bz) = \frac{\bz-\bs}{\rho}$, we have $\mathcal G_\rho(\bz)-G(\bz)\in\partial \delta_Z(\bs)$. As a result, convexity of $\delta_Z$ gives, for every $\bu\in \mathcal B$, that
\begin{align}\label{eq: yhu4}
0\leq \langle\mathcal G_\rho(\bz)-G(\bz),\bs-\bu\rangle \iff \langle G(\bz), \bs-\bu \rangle \leq \langle \mathcal{G}_\rho(\bz), \bs-\bu \rangle.
\end{align}
Using $\bz-\bs=\rho \mathcal G_\rho(\bz)$, we obtain
\begin{align}
\langle G(\bz),\bz-\bu \rangle 
&= \rho\langle G(\bz),  \mathcal{G}_{\rho}(\bz) \rangle+\langle G(\bz),\bs-\bu \rangle \notag \\
&\leq \rho \langle G(\bz), \mathcal G_\rho (\bz) \rangle
+\langle \mathcal G_\rho(\bz), \bs-\bu\rangle \notag\\
&\leq \left(\|\bs-\bu\|+\rho\|G(\bz)\|\right) \|\mathcal G_\rho(\bz)\|, \label{eq: rgap}
\end{align}
where the first inequality is by \eqref{eq: yhu4} and the second by Cauchy-Schwarz.

Since $\bs = \proj_C (\bz-\rho G(\bz)) \in C$, for every $\bu\in\mathcal U$, we have
\begin{align*}
\|\bs-\bu\| &\leq \|\bs-\bz\|+\|\bz-\bz_0\|+\|\bz_0-\bu\|\\
&= \rho\|\mathcal G_\rho(\bz)\| +\|\bz-\bz_0\| +\|\bz_0-\bu\|,
\end{align*}
where the first identity is because of $\bz-\bs=\rho \mathcal G_\rho(\bz)$.
Lipschitzness of $G$ gives $\|G(\bz)\| \leq \|G(\bz_0)\|+L\|\bz-\bz_0\|$. Substituting these inequalities into \eqref{eq: rgap} proves \eqref{eq: gapproj}. 
\end{proof}

The following result shows that the natural residual (gradient mapping norm) is upper bounded by a residual-type quantity. Moreover, we prove that the natural residual is Lipschitz, which is another standard property.
\begin{lemma}\label{lem:mapping-comparison}
Under \Cref{asp: 1}, for $\bz\in\dom\partial r$ and any
$\bs\in\partial r(\bz)$, any $\rho>0$ and for all $\bx,\by\in \dom r$, we have
\begin{align}
    \|\mathcal G_\rho(\bz)\|&\leq\|G(\bz)+\bs\|,\label{eq:mapping-tangent-comparison}
    \\
    \|\mathcal G_\rho(\bx)-\mathcal G_\rho(\by)\|
    &\leq \sqrt{\rho^{-2}+L^2}\|\bx-\by\|.
    \label{eq:mapping-lipschitz}
\end{align}
\end{lemma}

\begin{proof}
If $\bs\in\partial r(\bz)$, then
$\bz=\prox_{\rho r}(\bz+\rho\bs)$ since the latter is equivalent to $\bz+\rho\partial r(\bz) \ni \bz+\rho \bs$. Using this identity with nonexpansiveness gives
\[
    \rho\|\mathcal G_\rho(\bz)\| = \|\bz-\prox_{\rho r}(\bz-\rho G(\bz))\|
    =\|\prox_{\rho r}(\bz+\rho\bs)
                 -\prox_{\rho r}(\bz-\rho G(\bz))\|
    \leq\rho\|G(\bz)+\bs\|.
\]
We divide by $\rho$ 
to obtain \eqref{eq:mapping-tangent-comparison}.

For the second assertion, the definition of the proximal operator gives 
$\mathcal G_\rho(\bx)-G(\bx)
\in\partial r(\bx-\rho\mathcal G_\rho(\bx))$ since $\rho\mathcal{G}_\rho(\bx) = \bx-\prox_{\rho r}(\bx-\rho G(\bx))$,
and the analogous inclusion for $\by$.
Monotonicity of $\partial r$ implies
\begin{align*}
    \langle \mathcal{G}_\rho(\bx) - G(\bx) - [\mathcal{G}_\rho(\by) - G(\by)], \bx-\rho \mathcal{G}_\rho(\bx) - [\by-\rho\mathcal{G}_\rho(\by)] \rangle \geq 0.
\end{align*}
Expanding the inner product gives
\begin{align}\label{eq: huy4}
\rho\|\mathcal G_\rho(\bx)-\mathcal G_\rho(\by)\|^2
&\leq\inprod{\mathcal G_\rho(\bx)-\mathcal G_\rho(\by)}
{\bx-\by+\rho(G(\bx)-G(\by))}-\inprod{G(\bx)-G(\by)}{\bx-\by}.
\end{align}
Young's inequality gives
\begin{align*}
    \langle \mathcal{G}_\rho(\bx) - \mathcal{G}_\rho(\by), \bx-\by &+\rho(G(\bx) - G(\by)) \rangle \leq \frac{\rho}{2}\|\mathcal{G}_\rho(\bx) - \mathcal{G}_\rho(\by)\|^2 \\
    &\quad+ \frac{1}{2\rho}\left( \|\bx-\by\|^2 + 2\rho\langle \bx-\by, G(\bx)-G(\by) \rangle + \rho^2\| G(\bx)-G(\by)\|^2 \right).
\end{align*}
Plugging to \eqref{eq: huy4} yields
\begin{equation*}
\rho\|\mathcal G_\rho(\bx)-\mathcal G_\rho(\by)\|^2\leq\frac\rho2\|\mathcal G_\rho(\bx)-\mathcal G_\rho(\by)\|^2
+\frac{1}{2\rho}\|\bx-\by\|^2 +\frac\rho2\|G(\bx)-G(\by)\|^2.
\end{equation*}
After rearranging and using Lipschitz continuity of $G$, we deduce
\[
    \|\mathcal G_\rho(\bx)-\mathcal G_\rho(\by)\|^2
    \leq\rho^{-2}\|\bx-\by\|^2+\|G(\bx)-G(\by)\|^2
    \leq(\rho^{-2}+L^2)\|\bx-\by\|^2.
\]
Taking square root of both sides proves \eqref{eq:mapping-lipschitz}.
\end{proof}

\subsection{Numerical Consequences of Parameters}
The following fact collects some tedious, yet important estimations we used, which follow from the choice of our parameters.
\begin{fact}\label{fact: sces}
Suppose that $a\geq2$, $H^2= L^2+2B^2$, and
\begin{align*}
\beta_t=\frac{a}{t+a},
\qquad
\eta_t=\frac{1}{H(t+a)^{3/4}}.
\end{align*}
Then, for every $t\geq1$,
\begin{align}
\sqrt{(1-\beta_t)^2+\eta_t^2L^2}
\leq1-\frac23\beta_t.
\label{eq:boundedness-contraction}
\end{align}
Moreover, for every $t\geq0$, we have,
\begin{align}
\left(\frac{\eta_{t+1}}{\eta_t}-\beta_{t+1}\right)^2
+\eta_{t+1}^2L^2
\leq\left(1-\frac{3}{2(t+a+1)}\right)^2,
\label{eq: shj5}
\end{align}
and
\begin{align}
\left|\beta_{t+1}
-\frac{\eta_{t+1}}{\eta_t}\beta_t\right|
\leq\frac{a}{4(t+a)(t+a+1)}.
\label{eq: shj6}
\end{align}
Finally, for every $t\geq1$,
\begin{align}
(1-\beta_t)^2+\eta_t^2(L^2+2B^2)
\leq1-\frac{a}{t+a}.
\label{eq: tracking-coefficient}
\end{align}
\end{fact}
\begin{proof}
For the first estimate, we have
\begin{align*}
\left(1-\frac23\beta_t\right)^2
-(1-\beta_t)^2-\eta_t^2L^2
&=\frac23\beta_t-\frac59\beta_t^2-\eta_t^2L^2\\
&=\frac{1}{t+a}
\left(
\frac{2a}{3}
-\frac{5a^2}{9(t+a)}
-\frac{L^2}{H^2\sqrt{t+a}}
\right)\\
&\geq\frac{1}{t+a}
\left(
\frac{2a}{3}
-\frac{5a^2}{9(t+a)}
-\frac{1}{\sqrt{t+a}}
\right)\\
&\geq\frac{1}{t+a}
\left(
\frac{a(a+6)}{9(a+1)}
-\frac{1}{\sqrt{a+1}}
\right)
\geq0.
\end{align*}
Here, the first inequality uses $H\geq L$, and the second uses
$t\geq1$. The last inequality follows from
$a(a+6)\geq9\sqrt{a+1}$ for $a\geq2$. This proves
\eqref{eq:boundedness-contraction}.

We next prove the estimates for $t\geq0$. Since
\begin{align*}
\frac{\eta_{t+1}}{\eta_t}
=\frac{(t+a)^{3/4}}{(t+a+1)^{3/4}}
\geq\frac{t+a}{t+a+1}
\geq\beta_{t+1}
=\frac{a}{t+a+1},
\end{align*}
we have $\eta_{t+1}/\eta_t-\beta_{t+1}\geq0$.
Write $s=t+a+1\geq3$. By concavity of $x^{3/4}$, we have $\left(1-\frac1s\right)^{3/4} \leq 1-\frac{3}{4s}$.
Therefore, by definition of $\eta_t$ and $\beta_{t+1}$, we have
\begin{align*}
0 \leq\frac{\eta_{t+1}}{\eta_t}-\beta_{t+1} =
\left(1-\frac1s\right)^{3/4}-\frac{a}{s} \leq 1-\frac{3}{4s}-\frac{a}{s} \leq 1-\frac{11}{4s},
\end{align*}
where the last inequality uses $a\geq2$. Moreover, $\eta_{t+1}^2L^2
= \frac{L^2}{H^2}s^{-3/2} 
\leq s^{-3/2}$ because $H\geq L$. Therefore,
\begin{align*}
\left(\frac{\eta_{t+1}}{\eta_t} -\beta_{t+1}\right)^2
+\eta_{t+1}^2L^2
\leq \left(1-\frac{11}{4s}\right)^2+s^{-3/2}
\leq \left(1-\frac{3}{2s}\right)^2.
\end{align*}
The last inequality follows from $\frac52-\frac{85}{16s}-s^{-1/2} \geq \frac{35}{48}-\frac{1}{\sqrt3} >0$ because $s\geq3$. This proves \eqref{eq: shj5}.

Similarly, using the definitions of $\beta_t$ and $\eta_t$, we have
\begin{align*}
\left| \beta_{t+1} -\frac{\eta_{t+1}}{\eta_t}\beta_t \right|
&= \left|\frac{a}{t+a+1} - \left(\frac{t+a}{t+a+1} \right)^{3/4} \frac{a}{t+a} \right|\\
&= \frac{a}{t+a+1} \left| 1-\left(\frac{t+a+1}{t+a}\right)^{1/4} \right|\\
&= \frac{a}{t+a+1} \left[ \left(1+\frac{1}{t+a}\right)^{1/4}-1 \right].
\end{align*}
The last equality follows because
$\left(1+\frac{1}{t+a}\right)^{1/4}\geq1$.
Then, concavity of $x^{1/4}$ yields $\left(1+\frac{1}{t+a}\right)^{1/4} \leq 1+\frac{1}{4(t+a)}$. Therefore, we have
\begin{align*}
\left|\beta_{t+1}-\frac{\eta_{t+1}}{\eta_t}\beta_t\right|
\leq \frac{a}{4(t+a)(t+a+1)},
\end{align*}
which proves \eqref{eq: shj6}.

Finally, for $t\geq1$,
\begin{align}\label{eq: numer_est1}
(t+a)\left[1-(1-\beta_t)^2-\eta_t^2(L^2+2B^2)\right]
&= 2a-\frac{a^2}{t+a} - \frac{L^2+2B^2}{H^2\sqrt{t+a}}\notag \\
&\geq 2a-\frac{a^2}{t+a}-\frac{1}{\sqrt{t+a}}\geq a.
\end{align}
Here, the first inequality uses $H^2= L^2+2B^2$, while the last
inequality follows from
\begin{align*}
a-\frac{a^2}{t+a}-\frac{1}{\sqrt{t+a}}
\geq \frac{a}{a+1}-\frac{1}{\sqrt{a+1}}
\geq 0
\end{align*}
for $t\geq1$ and $a\geq2$. Dividing \eqref{eq: numer_est1} by $t+a$
proves \eqref{eq: tracking-coefficient}.
\end{proof}

\begin{fact}\label{fact: stes}
Suppose $H=\max\left\{2L,\;4\sqrt{1024\Lambda^2+4+16B^2}\right\}$ and let $\Gamma = 1024\Lambda^2+4+16B^2$. with
\begin{align*}
\alpha_t=\beta_t=\frac{1}{t+1}, ~~~\text{and}~~~\eta_t=\frac{1}{H\sqrt{t+1}}.
\end{align*}
Then for every $t\geq1$, following five estimates hold:
\begin{enumerate}
\item 
\begin{align}
&\left(1- \frac{1}{t+1}\right)^2\frac{\Gamma C}{t} +\frac{512\Lambda^2C}{(t+1)^2} + \frac{2(C+4B^2C)}{(t+1)^2} = \frac{\Gamma C}{t+1} -\frac{\Gamma C}{2(t+1)^2}.
\label{eq: fc0}
\end{align}

\item 
\begin{align}\label{eq: fc1}
\left(1-\beta_t+\frac{\eta_t^2L^2}{1-\beta_t} \right)
\leq \left(1-\frac{\beta_t}{2}\right).
\end{align}

\item 
\begin{align}\label{eq: fc2}
&40(t+1) \left(\frac{\eta_t}{\eta_{t-1}}\beta_{t-1}-\beta_t\right)^2C + \frac{10(t+1) \eta_t^2\alpha_t^2\Gamma C}{t}\notag\\
&\quad+ \eta_t^2 \left(\frac{512\Lambda^2C}{(t+1)^2} + 2\alpha_t^2\big(\sigma^2+4B^2C\big) \right) \leq \left(40 +\frac{41\Gamma}{2H^2} \right)\frac{C}{(t+1)^3}.
\end{align}

\item 
\begin{align}\label{eq: fc3}
40 +\frac{41\Gamma}{2H^2} \leq
\frac{1321}{32} \leq 64.
\end{align}
\item 
\begin{align}\label{eq: fc4}
&\frac{256C}{(t+1)^2}
\left(1+\frac{1}{4(t+1)}\right)
\left(\left(1- \frac{3}{2(t+1)} \right)^2+ \frac{1}{4(t+1)} \right)
+\frac{64C}{(t+1)^3} \notag\\
& \leq\frac{256C}{(t+2)^2}.
\end{align}
\end{enumerate}

\end{fact}
\begin{proof}
Throughout the proof, we use the notations
\begin{align*}
\Gamma = 1024\Lambda^2+4+16B^2.
\end{align*}
We start from the first statement, because $1-\frac{1}{t+1}=\frac{t}{t+1}$, the left-hand side satisfies
\begin{align*}
&\left(1- \frac{1}{t+1}\right)^2\frac{\Gamma C}{t} +\frac{512\Lambda^2C}{(t+1)^2} + \frac{2(C+4B^2C)}{(t+1)^2} \\
&=\frac{t\Gamma C}{(t+1)^2} +\frac{\Gamma C}{2(t+1)^2} \\
&=\frac{(t+1/2)\Gamma C}{(t+1)^2} \\
&= \frac{\Gamma C}{t+1} -\frac{\Gamma C}{2(t+1)^2},
\end{align*}
which proves \eqref{eq: fc0}.

Then we start the second statement. Using the schedule of $\beta_t$ and $\eta_t$, we have $\beta_t\leq1/2$ and $\eta_t^2L^2\leq\beta_t/4$. Using $\frac{1}{4(1-\beta_t)}\leq\frac12$, is suffices to show,
\begin{align*}
\left(1-\beta_t+\frac{\eta_t^2L^2}{1-\beta_t} \right)
\leq 1-\beta_t+\frac{\beta_t}{4(1-\beta_t)}
\leq \left(1-\frac{\beta_t}{2}\right),
\end{align*}
for every $t\geq1$. This finishes proof of \eqref{eq: fc1}.

We next move to the third statement. We estimate first term of the left-hand side, applying the definition of $\eta_t$ and $\beta_t$ gives,
\begin{align*}
40(t+1) \left(\frac{\eta_t}{\eta_{t-1}}\beta_{t-1}-\beta_t\right)^2C 
&= 40(t+1) \left(\frac{1}{\sqrt{t(t+1)}}-\frac{1}{t+1}\right)^2C \\
&= 40(t+1) \left(\frac{1}{\sqrt{t}(t+1)(\sqrt{t+1}+\sqrt{t})}\right)^2C
\end{align*}
Then, because $\sqrt{t}(\sqrt{t+1}+\sqrt{t})=\sqrt{t(t+1)}+t\geq1$, it implies
\begin{align*}
\frac{1}{\sqrt{t}(t+1)(\sqrt{t+1}+\sqrt{t})}
\leq \frac{1}{(t+1)^2}.
\end{align*}
Therefore, we have
\begin{align*}
40(t+1) \left(\frac{\eta_t}{\eta_{t-1}}\beta_{t-1}-\beta_t\right)^2C 
\leq \frac{40C}{(t+1)^3}.
\end{align*}

For the second term, using the definition of $\eta_t$ and $\alpha_t$,
\begin{align*}
\frac{10(t+1) \eta_t^2\alpha_t^2\Gamma C}{t} 
&= \frac{10\Gamma C}{H^2t(t+1)^2} \\
&\leq \frac{20\Gamma}{H^2}\frac{C}{(t+1)^3},
\end{align*}
where we used $1/t\leq2/(1+t)$ for the last inequality.

For the last term, substituting the $\eta_t$ and $\alpha_t$ yield
\begin{align*}
\eta_t^2 \left(\frac{512\Lambda^2C}{(t+1)^2} + 2\alpha_t^2\big(\sigma^2+4B^2C\big)\right)
&= \frac{1}{H^2(t+1)} \left(\frac{512\Lambda^2C}{(t+1)^2} + \frac{2(\sigma^2+4B^2C)}{(t+1)^2}\right) \\
&\leq \frac{\Gamma}{2H^2}\frac{C}{(t+1)^3},
\end{align*}
where the last inequality follows from $\sigma^2\leq C$.

Combining the three inequality give us
\begin{align*}
\frac{40C}{(t+1)^3} 
+ \frac{20\Gamma}{H^2}\frac{C}{(t+1)^3}
&+ \frac{\Gamma}{2H^2}\frac{C}{(t+1)^3}= \left(40 +\frac{41\Gamma}{2H^2} \right)\frac{C}{(t+1)^3}.
\end{align*}
Which proves the claim.

For the fourth statement, the definition of $H$ gives $H^2 \geq 16\Gamma$. Thus we have $\Gamma/H^2 \leq 1/16$, and consequently 
\begin{align*}
40+\frac{41\Gamma}{2H^2} 
\leq 40+\frac{41}{32}
= \frac{1321}{32}\leq 64.
\end{align*}
This proves the fourth statement.

It remains to prove last statement. Expanding the factor gives,
\begin{align}\label{eq: facexp}
\left(1+\frac{1}{4(t+1)}\right)
\left(\left(1- \frac{3}{2(t+1)} \right)^2+ \frac{1}{4(t+1)} \right) 
= \frac{16(t+1)^3-40(t+1)^2+25(t+1)+9}{16(t+1)^3}.
\end{align}

Subtracting the left-hand side of \eqref{eq: fc4} from its right-hand side yields
\begin{align*}
\frac{256C}{(t+2)^2}
&-\frac{256C}{(t+1)^2}
\left(1+\frac{1}{4(t+1)}\right)
\left(\left(1- \frac{3}{2(t+1)} \right)^2+ \frac{1}{4(t+1)} \right)
-\frac{64C}{(t+1)^3}\\
&= \frac{256}{(t+2)^2}
- \frac{16(16(t+1)^3-36(t+1)^2+25(t+1)+9)}{(t+1)^5} \\
&= \frac{16(4(t+1)^4+31(t+1)^3-23(t+1)^2-43(t+1)-9)}{(t+1)^5(t+2)^2},
\end{align*}
where the second equality uses \eqref{eq: facexp}, and the remaining equalities follow from direct calculation.

The numerator satisfies
\begin{align*}
4(t&+1)^4+31(t+1)^3-23(t+1)^2-43(t+1)-9 \\
&\quad =(t-1)(4(t+1)^3+39(t+1)^2+55(t+1)+67)+125>0 
\end{align*}
for every $t\geq1$. This completes the proof of the last statement.
\end{proof}

\end{document}